\documentclass[10pt,reqno,twoside]{amsart}
\usepackage{amssymb,amsmath,amsthm,soul,color,paralist}
\usepackage{t1enc}
\usepackage{comment}
\usepackage[cp1250]{inputenc}
\usepackage{a4,indentfirst,latexsym}
\usepackage{graphics}
\usepackage{mathrsfs}
\usepackage{cite,enumitem,graphicx}
\usepackage[colorlinks=true,urlcolor=blue,
citecolor=red,linkcolor=blue,linktocpage,pdfpagelabels,
bookmarksnumbered,bookmarksopen]{hyperref}
\usepackage[english]{babel}
\usepackage{geometry}
\usepackage[metapost]{mfpic}
\usepackage[hyperpageref]{backref}
\usepackage[colorinlistoftodos]{todonotes}
\usepackage[normalem]{ulem}

\allowdisplaybreaks 

\makeatletter
\providecommand\@dotsep{5}
\def\listtodoname{List of Todos}
\def\listoftodos{\@starttoc{tdo}\listtodoname}
\makeatother

\numberwithin{equation}{section}

\def\Xint#1{\mathchoice 
  {\XXint\displaystyle\textstyle{#1}}%
  {\XXint\textstyle\scriptstyle{#1}}%
  {\XXint\scriptstyle\scriptscriptstyle{#1}}%
  {\XXint\scriptscriptstyle\scriptscriptstyle{#1}}%
  \!\int} 
\def\XXint#1#2#3{{\setbox0=\hbox{$#1{#2#3}{\int}$} 
  \vcenter{\hbox{$#2#3$}}\kern-.5\wd0}} 
\def\-int{\Xint -}

\newcommand{\R}{\mathbb{R}}
\newcommand{\N}{\mathbb{N}}
\newcommand{\ri}{\rightarrow}

\DeclareMathOperator{\dive}{div}

\DeclareMathOperator{\B}{\mathcal{B}}
\DeclareMathOperator{\Q}{\mathcal{Q}}
\DeclareMathOperator{\A}{\mathcal{A}}
\DeclareMathOperator{\V}{\mathcal{\varphi}}
\DeclareMathOperator{\e}{\varepsilon}
\DeclareMathOperator{\dist}{dist_{\mathcal{P}}}

\newtheorem{lem}{Lemma}[section]
\newtheorem{thm}{Theorem}[section]
\newtheorem{defn}{Definition}[section]
\newtheorem{cor}{Corollary}[section]

\newtheorem{ex}{Example}[section]

\newtheorem{remark}{Remark}[section]
\newtheorem{assumption}[thm]{Assumption}

\begin{document}
\title[$\A$-caloric approximation and partial regularity]{$\A$-caloric approximation and partial regularity for parabolic systems with Orlicz growth}
\author[Foss]{Mikil Foss}
\author[Isernia]{Teresa Isernia}
\author[Leone]{Chiara Leone}
\author[Verde]{Anna Verde}

\address[M.\ Foss]{Department of Mathematics 
\newline\indent University of Nebraska-Lincoln
\newline\indent 203 Avery Hall, Lincoln, NE 68588-0130, USA.}
\email{mikil.foss@unl.edu}

\address[T.\ Isernia]{Dipartimento di Ingegneria Industriale e Scienze Matematiche
	\newline\indent
	Universit\`a Politecnica delle Marche
	\newline\indent
	Via Brecce Bianche 12, 60131 Ancona, Italy}
\email{t.isernia@univpm.it}

\address[C.\ Leone \& A.\ Verde]{Dipartimento di Matematica ``R. Caccioppoli''
\newline\indent
Universit\`a degli Studi di Napoli Federico II
\newline\indent
Via Cinthia, Complesso Universitario di Monte S. Angelo, 80126 Napoli, Italy}
\email{chiara.leone@unina.it}
\email{anna.verde@unina.it}

\subjclass[2010]{35B65, 35K40, 46E30}
\keywords{Parabolic systems, $\A$-caloric approximation, Orlicz growth, partial regularity}

\begin{abstract}
We prove a new $\A$-caloric approximation lemma compatible with an Orlicz setting. With this result, we establish a partial regularity result for parabolic systems of the type
$$
u_{t}- \dive a(Du)=0.
$$
Here the growth of $a$ is bounded by the derivative of an $N$-function $\V$. The primary assumption for $\V$ is that $t\V''(t)$ and $\V'(t)$ are uniformly comparable on $(0,\infty)$.
\end{abstract}

\maketitle

\section{Introduction}\label{S:Intro}

\noindent
In this paper, we establish a partial regularity result for weak solutions to parabolic systems with  general growth. By partial regularity, we mean H\"{o}lder continuity for the spatial gradient outside a closed set with zero measure. Let $\Omega \subset \R^{n}$ be an open bounded set, $n\geq 2$, $T>0$, and $N\geq 1$; we consider weak solutions $u:\Omega_{T}\ri \R^{N}$, where $\Omega_{T}= \Omega \times (-T, 0)$, to the following homogeneous  parabolic system 
\begin{equation}\label{P}
u_{t}- \dive a(Du)=0 \quad \mbox{ in } \Omega_{T}, 
\end{equation}
where the $C^{1}$-vector field $a:\R^{Nn}\ri \R^{Nn}$ satisfies ellipticity and growth conditions in terms of Orlicz functions. The precise structural assumptions on the vector field $a$ will be presented later, but the principal prototype we have in mind is the parabolic $\V$-Laplacian system
\begin{equation}\label{philaplacian}
u_{t}- \dive \left( \frac{\V'(\mu+|Du|)}{\mu+|Du|} Du\right)=0,
\end{equation}  
where $\mu> 0$ and $\V$ is an Orlicz function (see Subsection \ref{prel}). 
In the model case, with $\V(s)=\frac{1}{p}s^{p}$ for some $p>1$, \eqref{philaplacian} gives the more familiar non-degenerate evolutionary $p$-Laplacian system:
\begin{equation}\label{Pp}
u_{t}- \dive \left((\mu+|Du|)^{p-2} Du\right)=0.
\end{equation}
Hence system \eqref{philaplacian} (and consequently \eqref{P}) can be seen as a generalization of the $p$-Laplacian parabolic system \eqref{Pp}. In particular, in addition to not requiring the system~\eqref{P} to have the standard $p$-growth, we do not assume an Uhlenbeck structure as in\eqref{philaplacian}.

\noindent
The literature is rich with regularity results for parabolic systems with standard $p$-growth. In the paper by DiBenedetto and Friedman~\cite{DBF1}, everywhere regularity is proved. In this paper, the system has an Uhlenbeck structure: $a(\xi)=|\xi|^{p-2}\xi$ and $p>\frac{2n}{n+2}$. Using a combination of the Moser and De Giorgi iteration schemes, the solution's spatial gradient is shown to be bounded and H\"{o}lder continuous on its domain. In~\cite{DBF2}, the authors extended their result to allow nonlinear forcing and introduced the intrinsic scaling which accommodates the singular ($p<2$) or degenerate ($p>2$) behavior of $a$ in a natural way. (For a comprehensive introduction and collection of results on the subject, we refer the reader to DiBenedetto's book \cite{DB}.) It is well-known that, without special structural assumptions, solutions to systems can only be expected to possess partial regularity, that is regularity on an open set of full measure. Giaquinta and Giusti~\cite{GiaGiu} provided the first result in this direction. Adapting a blow-up argument, successfully used for elliptic systems, they showed partial H\"{o}lder continuity for the weak solution, $u$, of nondegenerate systems with $p$-growth ($p\ge2$). By again adapting techniques for elliptic systems, Giaquinta and Struwe proved higher integrability and partial H\"{o}lder continuity for a solution's spatial gradient, $Du$, provided $a$ has quadratic growth. For a general nonlinear $a(z,u,\xi)$ with quadratic growth, partial regularity for the spatial gradient remained an open problem until the work of Duzaar and Mingione~\cite{DM}. In this transformative paper, the authors introduced the, now well-known, $\mathcal{A}$-caloric approximation approach to regularity theory for parabolic systems. Generalizations to problems with superquadratic or subquadratic growth were provided by Duzaar, Mingione, and Steffen~\cite{DMS} and Scheven~\cite{scheven}. Utilizing intrinsic scaling and $p$-caloric approximation, along with $\mathcal{A}$-caloric approximation, B\"{o}gelein, Duzaar, and Mingione~\cite{BDM} extended these results to $p$-growth systems, of the form~\eqref{P}, that are potentially degenerate ($p>2$) or singular $\frac{2n}{n+2}<p<2$. Without any attempt for completeness, we also mention to the papers \cite{BDM, ba, BFM, FG, mons} where the partial H\"older continuity, either for the solution's spatial gradient or the solution itself, is established.

\smallskip

\noindent
The main goal of this paper is to extend some of these partial regularity results into the Orlicz-growth setting. In the papers cited above, the superquadratic $p\ge2$ and subquadratic $p<2$ cases require different techniques. Working in an Orlicz setting, we provide a unified treatment for both system classes.
\smallskip

\noindent
There is a long history of interest in partial differential equations with nonstandard growth. Early existence results for both elliptic and parabolic problems were established by Donaldson~\cite{Don1,Don2} (see also~\cite{Vis}). For elliptic equations and scalar-valued variational problems with $(p,q)$-growth, Marcellini~\cite{MarcARMA,MarcJDE} developed an approximation and Moser iteration technique to prove everywhere regularity. For elliptic systems with an Uhlenbeck structure, Marcellini and Papi~\cite{MarcPap} extended this strategy to even allow problems oscillating between linear and exponential growth (see also~\cite{MarcJMAA}). Under general growth conditions, additional results for elliptic systems can be found, for example, in~\cite{BV, CiaVlad1, CiaVlad2,CO, DE, DLSV, DSV, DSV1} and the references therein. Regarding regularity for parabolic systems with general growth, much less work is available in the literature. Assuming an Uhlenbeck structure, the iteration strategies developed in~\cite{DBF1,DBF2} have been adapted to problems with form~\eqref{P}. Assuming $t\V'(t)$ and $\V(t)$ are comparable on $(0,\infty)$, Lieberman~\cite{Lieb} proved that a weak solution $u$ to~\eqref{philaplacian} has a H\"older continuous spatial gradient $Du$, provided $Du$ is already known to be bounded. In~\cite{You}, You removed the boundedness assumption, but under stricter growth assumptions. More recently, the boundedness of $|Du|$ has been established under more general conditions. In~\cite{DSS}, Diening, Scharle, and Schwarzacher assume $t\V''(t)$ and $\V'(t)$ to be comparable, while in~\cite{Isernia}, only a doubling property is needed to obtain the boundedness of $u$. For additional regularity and higher integrability results, where an Uhlenbeck or similar structure is assumed, we also mention~\cite{BD,BDMar,BK,DSSV}. Without such a structural assumption, higher integrability was established by H\"{a}st\"{o} and Ok in~\cite{ho}. As far as the authors are aware, the current paper is the first to establish the partial H\"{o}lder continuity of a weak solution's spatial gradient.

\smallskip

\noindent
We now list the specific assumptions needed (see also Section \ref{prel}).

\begin{assumption}\label{assumption} 
Let $\V\in C^{1}([0, \infty))\cap C^{2}(0, \infty)$ be an $N$-function satisfying
\begin{align*}
0<p_{0}-1 \leq \inf_{t>0} \frac{t \V''(t)}{\V'(t)}\leq \sup_{t>0} \frac{t \V''(t)}{\V'(t)}\leq p_{1}-1, 
\end{align*} 
with $\frac{2n}{n+2}<p_{0}\leq p_{1}$. Without loss of generality we can assume that $p_{0}<2<p_{1}$. 
\end{assumption}

\noindent
With this $\V$, we consider \eqref{P} under the following hypotheses on the $C^1$-vector field $a: \R^{Nn}\rightarrow \R^{Nn}$: 
\smallskip
 
\begin{compactenum}
\item [$(a_{1})$] there exists $L>0$ such that 
\begin{align*}
&|a(\xi)|\leq L \V'(1+|\xi|) 
\end{align*}
holds for every $\xi \in \R^{Nn}$;
\item [$(a_{2})$] there exists $\nu>0$ such that 
\begin{align*}
Da(\xi)(\eta,\eta)\ge\nu\V''(1+|\xi|)|\eta|^2
\end{align*}
holds for any $\xi, \eta \in \R^{Nn}$;
\item [$(a_{3})$] for every $\xi \in \R^{Nn}$  
\begin{align*}
|Da(\xi)|\le L\V''(1+|\xi|);
\end{align*}
\item [$(a_{4})$] there exists a nondecreasing and concave function $\omega: [0, \infty) \ri [0, 1]$ with $\omega(0)=0$ such that 
\begin{align*}
|Da(\xi)-Da(\eta)|\le L\omega\left(\frac{|\xi-\eta|}{1+|\xi|+|\eta|}\right)\V''(1+|\xi|+|\eta|) 
\end{align*}
for every $\xi,\eta \in \R^{Nn}$. 
\end{compactenum}
While it ensures $t\V''(t)$ is comparable $\V'(t)$ for $t>0$, Assumption~\ref{assumption} does not imply $\V$ has $p$-growth. It does imply $\V$ has the doubling property~\eqref{E:dbling}, and therefore, there exists a $p_0\le p\le p_1$ and a $C<\infty$ such that for each $\varepsilon>0$
\begin{equation*}
\lim_{t\to\infty} \frac{\V(t)}{t^{p+\varepsilon}}= \lim_{t\to\infty} \frac{t^{p-\varepsilon}}{\V(t)}=0.
\end{equation*}
Similar assumptions also appear in several of the works cited above. 
\smallskip

\noindent
The notion of weak solution adopted in the present paper is the following: $u\in C^0(-T,0, L^{2}(\Omega, \R^{N}))\cap L^{1}(-T,0, W^{1, 1}(\Omega,\R^N))$ with $\V(|Du|)\in L^{1}(-T,0, L^{1}(\Omega))$ is a weak solution to \eqref{P} if it holds 
\begin{equation*}
\int_{\Omega_T} u\cdot\eta_{t} - a( Du)\cdot D\eta \,dx\,dt =0,
\end{equation*}
for all $\eta \in C^{\infty}_{c}(\Omega_T, \R^{N})$. Here $Du:(-T,0)\times\Omega\to\R^{Nn}$ denotes the spatial gradient of $u$.
\smallskip

\noindent
We can now state our regularity result. 
\begin{thm}\label{main}
Let $u\in C^0(-T,0; L^{2}(\Omega, \R^{N}))\cap L^{1}(-T,0; W^{1, 1}(\Omega,\R^N))$ with $\V(|Du|)\in L^1(-T,0;L^1(\Omega))$ be a weak solution to \eqref{P} in $\Omega_T$ under hypotheses $(a_{1})$-$(a_{4})$ and Assumption \ref{assumption}. Then for every $\alpha\in (0,1)$ there exists an open subset $\Omega_0\subseteq\Omega_T$ such that 
$$
V(Du)\in C_{loc}^{0,\frac{\alpha}{2},\alpha}(\Omega_0,\R^{Nn}) \qquad \hbox{ and }\qquad |\Omega_T\setminus\Omega_0|=0,
$$
where $V(\xi)=\sqrt{\frac{\V'(1+|\xi|)}{1+|\xi|}}\xi$.
Moreover, the singular set $\Omega_T\setminus\Omega_0\subseteq\Sigma_1\cup\Sigma_2$, where 
\begin{align*}
&\Sigma_1=\left\{z_0\in\Omega_T : \liminf_{\rho\to 0}\-int_{\Q_\rho(z_0)}|V(Du)-(V(Du))_{z_0,\rho}|^2dz>0\right\},\\
&\Sigma_2=\left\{z_0\in\Omega_T : \limsup_{\rho\to 0}|(Du)_{z_0,\rho}|=+\infty\right\},
\end{align*}
denoting the mean value of a function over the parabolic cylinder $\Q_\rho(z_0)=\B_\rho(x_0)\times (t_0-\rho^2,t_0)$ by $(\cdot)_{z_0,\rho}$.
\end{thm}
\noindent
Here $z_0=(x_0,t)\in(-T,0)\times\Omega$ and $\B_\rho(x_0)$ is the ball in $\R^n$ with radius $\rho$ centered at $x_0$. Note that the H\"older continuity of $V(Du)$ implies the H\"older continuity of $Du$ with a different exponent depending on $\V$. 

\noindent
The proof of Theorem \ref{main} relies on a decay estimate for certain excess functionals, which measure in a suitable way the oscillations of the solutions. More precisely, for $z_{0}\in \Omega_{T}$, $r>0$, $a\geq 0$, and an affine map $\ell: \R^{n}\ri \R^{N}$, we define the excess functional by 
\begin{equation*}
\Psi_{a}(z_{0}, r, \ell)= \-int_{\Q_{r}(z_{0})} \left( \left| \frac{u-\ell}{r}\right|^{2} + \V_{a} \left( \left| \frac{u-\ell}{r}\right| \right) \right) \, dz.
\end{equation*}
Here $\V_a$ interpolates in a certain sense between $t^2$, when $t\le a$, and $\V$, when $t\ge a$. (The precise definition of the function $\V_a$ is given in Section \ref{prel}.)

\noindent
In order to achieve the decay estimate, we first derive the Caccioppoli inequality which is compatible with \eqref{P} (see Theorem \ref{cacc2}). This, in particular, allows us to control the spatial oscillations of $u$ and oscillations in $Du$ via the excess functional. Though $u$ need not be differentiable, if $\Psi_a$ is sufficiently small, then a family of smooth approximations to a spatial linearization of $u$, centered at $z_0$, can be produced. These 1st-order surrogates for $u$ are, in fact, solutions to a constant coefficient parabolic system. Moreover, their approximation to $u$ improves as $r\to0^+$ providing a decay estimate for the $\Psi_a$, which implies the oscillations in $V(Du)$ decrease as $r\to0^+$. The rate of decrease is fast enough to deliver the regularity of $V(Du)$ through Campanato's characterization of H\"{o}lder continuity.

\noindent
The outline of the proof of Theorem~\ref{main} follows the approach developed in~\cite{DMS} and~\cite{scheven}. The cornerstone to the strategy is the $\A$-caloric approximation theorem, which provides the family of approximations to $u$. The generalization of this theorem to something suitable for the Orlicz setting was a significant obstacle and is the paper's principal novelty. Our proof for this result does not require Assumption~\ref{assumption}. In fact, $\V$ is only assumed to be an $N$-function with super $\frac{2n}{n+2}$-growth and a doubling property near zero. Thus $\V$ may have exponential or even super-exponential growth. A key difference between the $p$-growth and Orlicz settings is that for $L^p$-spaces one has
\[
    L^p(\Q_\rho)=L^p(-\rho^2,0;L^p(\B_\rho)),
\]
but for the Orlicz spaces $L^{\V}$ one only has
\[
    L^{\V}(\Q_\rho)\subseteq L^{\V}(-\rho^2,0;L^{\V}(\B_\rho)).
\]
In fact, equality holds if and only if $\V(t)$ is comparable to $t^p$ (see the remarks following Proposition 1.3 in~\cite{Don2}). With standard growth, the proofs for the $\mathcal{A}$-caloric approximation theorem can take advantage of Simon's compactness result~\cite{Simon} in $L^p((-\rho^2;L^p(\B_\rho))$ to directly obtain convergence in $L^p(\Q_\rho)$. This, however, is not possible in the Orlicz setting. While we use Simon's result for convergence in $L^2$, upgrading to convergence in $L^{\V}(\Q_\rho)$ involves a combination of approximations via convolution, sophisticated pointwise estimates, and integral bounds for the non-centered Hardy-Littlewood maximal function. With this new $\mathcal{A}$-caloric approximation, we prove a decay estimate for the excess function. Employing a standard iteration argument, we are able to identify the singular set with points where either the excess cannot be made sufficiently small, $\Sigma_1$, or the mean $(Du)_{z_0,\rho}$ is not bounded, $\Sigma_2$. Finally, to prove that the singular set is negligible, we use a Poincar\'e-Sobolev-type inequality for solutions to~\eqref{P} which bounds the excess of $u$ in terms of its spatial gradient $Du$. The proof for this inequality is rather complicated and relies on a Gagliardo-Nirenberg inequality from~\cite{ho}.


\smallskip

\noindent
The paper is organized as follows: after collecting the basic terminology and other preliminaries in Section~\ref{prel}, we present the $\A$-caloric approximation in Section~\ref{Acal}. In Section~\ref{Caccioppoli}, the proofs of the Cacciopoli inequalities in the parabolic setting. We detail the Poincar\'e-Sobolev-type inequalities in Section~\ref{poinc} and the linearization in Section~\ref{lin}. We finally establish the decay estimates and the main theorem in the last two sections.


\section{Notation and Preliminary Results}\label{prel}

\noindent
Let $\Omega \subset \R^{n}$ be a bounded domain; in the following $\Omega_{T}$ will denote the parabolic
cylinder $\Omega \times (-T, 0)$, where $T>0$. If $z\in \Omega_{T}$, we denote $z=(x,t)$ with $x\in \Omega$
and $t\in (-T, 0)$. In what follows $C$ will be often a general positive constant, possibly varying from line to line, but depending on only the structural parameters $n,N,L/\nu,p_0,p_1$, with $1<p_0,p_1<\infty$ identified in Assumption~\ref{assumption} below. The notation $Du(x,t)\equiv D_{x} u(x,t)$ denotes the differentiation with respect to the spatial variable $x$, and $u_{t}$ stands for the differentiation with respect to the time variable. 

\noindent
With $x_{0}\in \R^{n}$, we set 
$$
\B_{r}(x_{0}):=\{x\in \R^{n} : |x-x_{0}|<r\}
$$
the open ball of $\R^{n}$ with radius $r>0$ and center $x_{0}$. When dealing with parabolic regularity, the geometry of cylinders plays an important role. 
We denote the general cylinder with spatial radius $\rho$ and time length $\tau$ centered at $z_0 =(x_0,t_0)$ by
$$
\Q_{\rho,\tau}(z_0)=\B_\rho(x_0)\times (t_0-\tau,t_0),
$$
and we define the standard parabolic cylinder  by
$$
\Q_{\rho}(z_0)=\Q_{\rho,\rho^2}(z_0)=\B_\rho(x_{0})\times (t_{0}-\rho^{2}, t_{0}). 
$$
Given a cylinder $\Q=\B\times (s,t)$, its parabolic boundary is 
$$
\partial_{\mathcal{P}}\!\Q :=(\B \times \{s\})\cup (\partial\!\B \times [s,t]).
$$
The integral averages of a function $u$ on $\Q\subset \R^{n+1}$ are given by
$$
(u)_{\Q}(t)=\-int_{\Q} u(x, t) dx, \ \ \  (u)_{\Q}=\-int_{\Q} u(x,t)\,dz.
$$
We will denote the average $(u)_{\Q_\rho(z_0)}$ by $(u)_{z_0,\rho}$.
The parabolic metric is defined as usual by 
$$
\dist (z, z_{0}):= \sqrt{|x-x_{0}|^{2}+|t-t_{0}|} 
$$
whenever $z=(x,t),z_{0}=(x_{0},t_{0})\in \R^{n+1}$. 

\bigskip

\noindent
We recall that a strongly elliptic bilinear form $\mathcal{A}$ on $\R^{Nn}$ with ellipticity constant $\nu>0$ and upper bound $L>0$ means that 
\begin{equation*}
\nu |\xi|^2 \leq \mathcal{A}(\xi, \xi), \quad \mathcal{A}(\xi, \tilde{\xi}) \leq L|\xi||\tilde{\xi}| \quad \forall \xi,\tilde{\xi}\in \R^{Nn}. 
\end{equation*}
\begin{defn}
We shall say that a function $h\in L^{2}(t_0-\rho^2,t_0; W^{1, 2}(\B_{\rho}(x_0), \R^{N}))$ is $\mathcal{A}$-caloric on $\Q_{\rho}(z_0)$ if it satisfies
\begin{equation*}
\int_{\Q_{\rho}(z_0)} h\cdot \eta_{t} - \mathcal{A}(Dh, D\eta)\, dz=0, \quad \mbox{ for all } \eta\in C^{\infty}_{c}(\Q_{\rho}(z_0), \R^{N}). 
\end{equation*}
\end{defn}

\begin{remark}\label{ut}
In the following we shall often write $u_t$ even if a weak solution of a parabolic system may not be differentiable in the time variable.  The arguments can be made
rigorous by the use of a smoothing procedure in time, as for instance via Steklov averages. However, since this argument is by now quite standard, we shall abuse the notation $u_t$ proceeding formally, without further explanation.
\end{remark}


\medskip

\subsection{$N$-functions}

We begin recalling the notion of $N$-functions (see \cite{raoren}).

\noindent
We write $f\sim g$, and we say that $f$ and $g$ are equivalent, if there exist constants $c_{1}, c_{2} >0$ such that $c_{1}g(t) \leq f(t) \leq c_{2}g(t)$ for any $t\ge 0$. Similarly the symbol $\lesssim$ stands for $\le$ up to a constant.

\begin{defn}
A real convex function $\varphi: [0, \infty)\rightarrow [0, \infty)$ is said to be an $N$-function if $\varphi(0)=0$ and there exists a right continuous nondecreasing derivative $\V'$ satisfying $\varphi'(0)=0$, $\varphi'(t)>0$ for $t>0$ and $\displaystyle{\lim_{t\rightarrow \infty} \varphi'(t)=\infty}$. 
\end{defn}

\noindent 
An $N$-function $\V$ satisfies the $\Delta_{2}$-condition, and we write $\V\in \Delta_{2}$, if there exists a constant $c>0$ such that  
\begin{equation}\label{E:dbling}
\varphi(2t)\leq c\,\varphi(t) \quad \mbox{ for all } t\geq 0.
\end{equation}
The smallest possible constant will be denoted by $\Delta_{2}(\V)$. Combining $\V(t)\leq \V(2t)$ together with the $\Delta_{2}$-condition we get $\V(2t)\sim \V(t)$. 

\noindent
The {\it conjugate function} $\V^{*}:[0, +\infty) \ri [0, +\infty)$ of an $N$-function $\V$ is defined by 
\begin{align*}
\V^{*}(t) := \sup_{s\geq 0}\, [s t -\V(s)] \quad \mbox{ for all } t\geq 0.
\end{align*}
It holds that $\V^{*}$ itself is an $N$-function. 
If $\V$ and $\V^{*}$ both satisfy the $\Delta_{2}$-condition, then we will write that $\Delta_{2}(\V, \V^{*}):=\max\{\Delta_2(\V),\Delta_2(\V^*)\}<\infty$. 
Assume that $\Delta_{2}(\V, \V^{*})<\infty$; then for all $\delta>0$ there exists $c_{\delta}$ depending only on $\Delta_{2}(\varphi, \varphi^{*})$ such that
for all $s, t\geq 0$ it holds the Young's inequality
\begin{align*}
&t\, s\leq \delta \, \varphi(t)+c_{\delta} \, \varphi^{*}(s). 
\end{align*}

\noindent In most parts of the paper 
 we will assume that $\V$ satisfies Assumption \ref{assumption}.


\begin{remark}
We remark that the lower bound on $p_0$ appearing in Assumption \ref{assumption} is absolutely natural to prove regularity in such a context:  one need only to consider the power case $\V(t)=t^{p}$ (see \cite{DB}).
\end{remark}

\noindent
Under Assumption \ref{assumption} on $\V$ it follows from \cite[Proposition 2.1]{CO} that $\Delta_2(\V,\V*)<\infty$ and 
\begin{align*}
p_{0}\leq \inf_{t>0} \frac{t\V'(t)}{\V(t)}\leq \sup_{t>0} \frac{t\V'(t)}{\V(t)}\leq p_{1}. 
\end{align*}

\noindent
We point out that the following inequalities hold for every $t\geq 0$:  
\begin{align}\begin{split}\label{t1}
&s^{p_{1}} \V(t) \leq \V(st) \leq s^{p_{0}} \V(t) \quad \mbox{ if } 0< s \leq 1, \\
&s^{p_{0}} \V(t) \leq \V(st) \leq s^{p_{1}} \V(t) \quad \mbox{ if } s \geq 1,
\end{split}\end{align}
as well
\begin{align}\begin{split}\label{t1conj}
&s^{\frac{p_0}{p_0-1}} \V^*(t) \leq \V^*(st) \leq s^{\frac{p_{1}}{p_1-1}} \V^*(t) \quad \mbox{ if } 0< s \leq 1, \\
&s^{\frac{p_{1}}{p_1-1}} \V^*(t) \leq \V^*(st) \leq s^{\frac{p_{0}}{p_0-1}} \V^*(t) \quad \mbox{ if } s \geq 1,
\end{split}\end{align}
and also 
\begin{align} \begin{split} \label{t2}
&s^{p_{1}-1} \V'(t) \leq \V'(st) \leq s^{p_{0}-1} \V'(t) \quad \mbox{ if } 0< s \leq 1, \\
&s^{p_{0}-1} \V'(t) \leq \V'(st) \leq s^{p_{1}-1} \V'(t) \quad \mbox{ if } s \geq 1. 
\end{split}
\end{align}
In particular, for $t>0$ we have 
\begin{align}\label{t3}
\V(t)\sim t\V'(t), \quad \V'(t)\sim t\V''(t), \quad \V^{*}(\V'(t)) \sim \V(t), \quad \V^{-1}(t)(\V^{*})^{-1}(t)\sim t. 
\end{align}

\noindent
Now, we consider a family of $N$-functions $\{\varphi_{a}\}_{a\geq 0}$ setting, for  $t\geq 0$,
\begin{equation*}
\varphi_{a}(t):=\int_{0}^{t} \varphi'_{a}(s) \, ds \quad \mbox{ with } \quad \varphi'_{a}(t):= \varphi'(a+t) \frac{t}{a+t}. 
\end{equation*}
\noindent
The following lemma can be found in \cite[Lemma 27]{DE}.

\begin{lem}\label{phia}
Let $\V$ be an $N$-function with $\V\in\Delta_2$ together with its conjugate. Then for all  $a\geq 0$ the function $\varphi_a$ is an $N$-function and $\{\varphi_{a}\}_{a\geq 0}$ and $\{(\varphi_{a})^{*}\}_{a\geq 0}\sim\{\varphi^*_{\varphi'(a)}\}_{a\geq 0}$ satisfy the $\Delta_{2}$-condition
uniformly in $a\geq 0$.
\end{lem}

\noindent
Let us observe that by the previous lemma $\varphi_{a}(t)\sim t\varphi'_{a}(t)$. Moreover, for $t\geq a$ we have $\varphi_{a}(t)\sim \varphi(t)$,  while Assumption~\ref{assumption}
provides
\begin{equation}\label{phiaquadratic}
    C^{-1}\V''(a)t^2\le\V_a(t)
    \le
    C\V''(a)t^2, \quad \hbox{ for } 0\le t\le a,
\end{equation}
since $\V'\sim t\V''$. The constant $C$ depends on only $p_0$ and $p_1$.
This implies also  that, for all $s\in [0,1]$, $a\geq 0$ and $t\in [0,a]$, 
\begin{equation}\label{phiaquadratic2}
\varphi_{a}(st)\le C^2s^{2} \varphi_{a}(t).
\end{equation}

\noindent Finally, allowing Assumption~\ref{assumption}, the following relations hold uniformly with respect to $a\geq 0$
\begin{align}\begin{split}\label{t6}
&\V_{a}(t)\sim \V''(a+t) t^{2} \sim \frac{\V(a+t)}{(a+t)^{2}}t^{2} \sim \frac{\V'(a+t)}{a+t}t^{2}, \\
&\V(a+t)\sim \V_{a}(t) + \V(a). 
\end{split}\end{align}

\begin{remark}\label{phiap0p1}
It is easy to check that if $\V$ satisfies Assumption \ref{assumption}, the same is true for $\V_a$, uniformly with respect to $a\geq 0$ (with the same $p_0$ and $p_1$). This in particular means that 
\begin{equation}\label{E:ShiftDouble}
\sup_{a\ge 0}\Delta_2(\V_a)\le 2^{p_1}
\end{equation}
and
\begin{equation*}
t^{p_0}\V_a(1)\le\V_a(t)\le t^{p_1}\V_a(1), \quad \hbox{ for } t\ge a\ge 1,
\end{equation*}
thanks to (\ref{t1}) for $\V_a$.
\end{remark}

\noindent
Next result is a slight generalization of \cite[Lemma 20]{DE}. 
\begin{lem}\label{lem2}
Let $\varphi$ be an $N$-function with $\Delta_{2}(\varphi,\varphi*)<\infty$; then, uniformly in $\xi_{1}, \xi_{2} \in \R^{Nn}$  with $|\xi_{1}|+|\xi_{2}|>0$, and in $\mu\geq 0$, it holds
\begin{equation*}
 \frac{\varphi'(\mu+|\xi_{1}|+|\xi_{2}|)}{\mu +|\xi_{1}|+|\xi_{2}|} \sim \int_{0}^{1} \frac{\varphi'(\mu+|\xi_{\theta}|)}{\mu+|\xi_{\theta}|} d\theta, 
\end{equation*} 
where $\xi_{\theta}= \xi_{1} + \theta (\xi_{2}-\xi_{1})$ with $\theta \in [0,1]$.
\end{lem}

\begin{remark}\label{rem1}
We now state the following two consequences of our structure assumptions for further reference. First, we note that the ellipticity condition $(a_2)$ and  Assumption \ref{assumption} together with Lemma \ref{lem2} and \eqref{t6} imply
\begin{equation*}
(a(\xi)-a(\xi_0))\cdot (\xi-\xi_0)\ge c \V_{1+|\xi_0|}(|\xi-\xi_0|),
\end{equation*}
for every $\xi,\xi_0\in\R^{Nn}$, where $c=c(p_0,p_1,\nu)$.\\
In a similar way, the growth condition $(a_3)$ and Assumption \ref{assumption} imply
\begin{equation*}
|a(\xi)-a(\xi_0)|\le c \V'_{1+|\xi_0|}(|\xi-\xi_0|),
\end{equation*}
for every $\xi,\xi_0\in\R^{Nn}$, where $c=c(p_0,p_1,L)$.
\end{remark}

\noindent The following results deal with the change of shift of $N$-functions $\varphi_{a}$. The first one is proved in \cite[Corollary 26]{DK}.

\begin{lem}\label{shift}
Let $\V$ be a $N$-function satisfying $\Delta_2(\V,\V^*)<+\infty$,. Then for each $\delta>0$, there exists $c_\delta(\Delta_2(\V,\V^*)) > 0$ such that for all $a,b\in\R^d$ and $t\ge 0$
$$
\V_{|a|}(t)\le c_\delta \V_{|b|}(t) + \delta \V_{|a|}(|a-b|).
$$
\end{lem}

\begin{lem}\label{L:ShiftComp}
Let $\V$ be a $N$-function satisfying Assumption \ref{assumption}; let $M\ge 1$ and $1\le a,b\le M$ be given. Then
\begin{equation*}
    \V_a(t)\le 4^{p_1+1}M^{p_1+2}\V_b(t),
    \quad\text{ for all }t\ge 0.
\end{equation*}
\end{lem}
\begin{proof}
From the definition and the fact that $t/2\V'(t/2)\le\V(t)\le t\V'(t)$,
\begin{equation*}
    \V'_a(s)=\V'_b(s)\left[\frac{\V'(a+s)}{\V'(b+s)}\left(\frac{s+b}{s+a}\right)\right]
    \le
    \V'_b(s)\left[\frac{\V(2(a+s))}{\V(b+s)}
    \left(\frac{s+b}{s+a}\right)^2\right].
\end{equation*}
If $s\le M$, then, using \eqref{t1}, we have
\begin{equation*}
    \frac{\V(2(a+s))}{\V(b+s)}\left(\frac{s+b}{s+a}\right)^2
    \le\frac{\V(4M)}{\V(1)}M^2
    \le 4^{p_1}M^{p_1+2}.
\end{equation*}
Otherwise, $1\le a,b\le M< s$ and
\begin{equation*}
    \frac{\V(2(a+s))}{\V(b+s)}\left(\frac{s+b}{s+a}\right)^2
    \le
    4\frac{\V(4s)}{\V(s)}\le4^{p_1+1}.
\end{equation*}
Thus
\begin{equation*}
    \V_a(t)=\int_0^t\V_a'(s)\, d s
    \le\left(4^{p_1}M^{p_1+2}+4^{p_1+1}\right)
        \int_0^t\V_b'(s)\, d s
    \le
    4^{p_1+1}M^{p_1+2}\V_b(t).
\end{equation*}
\end{proof}

\medskip

\noindent
We will use the function $V:\R^{Nn}\rightarrow \R^{Nn}$ defined by 
$$
V(\xi) = \sqrt{\frac{\V'(1+|\xi|)}{1+|\xi|}}\xi.
$$
The monotonicity property of $\V$ ensures that 
\begin{align}\label{equiVphi}
|V(\xi_{1})-V(\xi_{2})|^{2}\sim \V_{1+|\xi_{1}|}(|\xi_{1}-\xi_{2}|) \quad \mbox{ for any } \xi_{1}, \xi_{2} \in \R^{Nn};
\end{align}
 see \cite{DE} for further properties about the $V$-function.

\medskip

\noindent
Let $\V$ be an $N$-function that satisfies the $\Delta_{2}$-condition. The set of functions $L^{\V}(\Omega, \R^{N})$ is defined by
\begin{align*}
L^{\V}(\Omega, \R^{N})= \left\{ u: \Omega \ri \R^{N} \mbox{ measurable } : \, \int_{\Omega} \V(|u|)\, dx <\infty\right\}. 
\end{align*}
The Luxembourg norm is defined as follows:
\begin{align*}
\|u\|_{L^\varphi(\Omega, \R^{N})}=\inf \left\{\lambda>0 : \int_{\Omega} \varphi \left(\frac{|u(x)|}{\lambda} \right)\,dx\leq 1\right\}.
\end{align*}
With this norm $L^\varphi(\Omega, \R^{N})$ is a Banach space. 

\noindent
By $W^{1, \V}(\Omega, \R^{N})$ we denote the classical Orlicz-Sobolev space, that is $u\in W^{1, \V}(\Omega, \R^{N})$ whenever $u,  Du \in L^{\V}(\Omega, \R^{N})$.
Furthermore, by $W^{1,\V}_{0}(\Omega, \R^{N})$ we mean the closure of $C^{\infty}_{c}(\Omega, \R^{N})$ functions with respect to the norm
\begin{align*}
\|u\|_{W^{1, \V}(\Omega, \R^{N})}=\|u\|_{L^{\V}(\Omega, \R^{N})}+\| Du\|_{L^{\V}(\Omega, \R^{N})}. 
\end{align*}

\noindent For a function $u\in L^\varphi(\Q_\rho(z_0), \R^{N})$, using the decomposition
$$
\Q_\rho^{\le}(z_0)=\{z\in\Q_\rho(z_0) : |u(z)|\le a\}\quad  \hbox{ and } \quad \Q_\rho^{>}(z_0)=\{z\in\Q_\rho(z_0) : |u(z)|> a\},
$$
as well as \eqref{phiaquadratic} and Remark \ref{phiap0p1}, we easily get the following lemma.

\begin{lem}\label{L:OrliczPower}
Let $\V$ be a $N$-function satisfying Assumption \ref{assumption} and let $u\in L^{\V}(\Q_\rho(z_0),\R^N)$, $a\ge 1$. Then
\begin{compactenum}[(a)]
    \item 
\begin{equation*}
    \-int_{\Q_\rho(z_0)}\V_a(|u|)\, dz
    \le
    C\V''(a)\-int_{\Q_\rho(z_0)}|u|^2\, dz+\V_a(1)\-int_{\Q_\rho(z_0)}|u|^{p_1}\, dz;
\end{equation*}
    \item for each $0\le s\le p_0<2$,
\begin{equation*}
    \-int_{\Q_\rho(z_0)}|u|^s\, dz
    \le
    \left(\frac{C}{\V''(a)}\-int_{\Q_\rho(z_0)}\V_a(|u|)\, dz\right)^\frac{s}{2}
    +\frac{1}{\V_a(1)}\-int_{\Q_\rho(z_0)}\V_a(|u|)\, dz.
\end{equation*}
\end{compactenum}
Here $C$ depends on only $p_0,p_1$.
\end{lem}

\medskip

\subsection{Affine functions}

\noindent
Let $z_{0}\in \R^{n+1}$ and $\rho>0$. Given $u\in L^{2}(\Q_{\rho}(z_{0}), \R^{N})$, we denote by $\ell_{z_{0}, \rho}:\R^{n}\ri \R^{N}$ the unique affine function minimizing the functional
\begin{align*}
\ell(x)\mapsto \-int_{\Q_{\rho}(z_{0})} |u(x,t)-\ell(x)|^2\, dz
\end{align*}
amongst all affine functions $\ell: \R^{n}\ri \R^{N}$. It is well known (see \cite{BDM}) that 
\begin{equation}\label{E:AffineMin}
\ell_{z_{0}, \rho}(x)= (u)_{z_{0}, \rho} + P_{z_{0}, \rho}(x-x_{0}), 
\end{equation}
where 
\begin{align}\label{gradaffine}
P_{z_{0}, \rho}= \frac{n+2}{\rho^{2}} \-int_{\Q_{\rho}(z_{0})} u(x,t)\otimes (x-x_{0})\, dz.  
\end{align}
\noindent The following lemma ensures that $\ell_{z_0,\rho}$ is an almost minimizer of the functional $\displaystyle{\ell\mapsto \-int_{\Q_{\rho}(z_{0})} \V\left(\frac{|u-\ell|}{r}\right)\, dz}$ amongst the affine functions $\ell: \R^{n}\ri \R^{N}$.

\begin{lem}\label{quasimin}
Let $\V$ be an $N$-function satisfying the $\Delta_2$-property and let $u\in L^{\V}(\Q_{\rho}(z_{0}), \R^{N})$. Let $r>0$, then there exists a constant $\kappa_0= \kappa_0(n,\Delta_2(\V))>0$ such that
$$
 \-int_{\Q_{\rho}(z_{0})} \V\left(\frac{|u-\ell_{z_0,\rho}|}{r}\right)\, dz\le \kappa_0 \-int_{\Q_{\rho}(z_{0})} \V\left(\frac{|u-\ell|}{r}\right)\, dz,
 $$
 for every  affine function $\ell: \R^{n}\ri \R^{N}$.
\end{lem}

\begin{proof}
Assume $z_{0}=(0,0)$ and denote $\ell_{z_0,\rho}$, $\Q_\rho(z_0)$, and $(u)_{z_0,\rho}$ by $\ell_{\rho}$, $\Q_\rho$, and $(u)_{\rho}$, respectively. Let us consider a generic affine function $\ell(x)=\zeta+Ax$, then, for  $x\in \B_\rho$, 
$$
|\ell-\ell_\rho|=|(u)_\rho-\zeta+ (D\ell_\rho-A)x|\le |(u)_\rho-\zeta|+\rho|D\ell_\rho-A|.
$$
Now we have 
$$
|(u)_\rho-\zeta|=\left|\-int_{\Q_\rho}(u-\zeta)\,dz\right|=\left|\-int_{\Q_\rho}(u-\zeta-Ax)\,dz\right|\le \-int_{\Q_\rho}|u-\ell|\,dz,
$$
and, using \eqref{gradaffine},
\begin{equation}\label{touse}
|D\ell_\rho-A|=\frac{n+2}{\rho^2}\left|\-int_{\Q_\rho}(u-Ax)\otimes x\,dz\right|=\frac{n+2}{\rho^2}\left|\-int_{\Q_\rho}(u-\zeta-Ax)\otimes x\,dz\right|\le \frac{n+2}{\rho}\-int_{\Q_\rho}|u-\ell|\,dz.
\end{equation}
In conclusion 
\begin{equation}\label{l-lrho}
|\ell-\ell_\rho|\le (n+3)\-int_{\Q_\rho}|u-\ell|\,dz.
\end{equation}
Recalling that, by the convexity and the $\Delta_2$-condition, $\V(s+t)\sim \V(s)+ \V(t)$ for any $s, t\geq 0$, we have 
\begin{align*}
\-int_{\Q_{\rho}} \V\left(\frac{|u-\ell_{\rho}|}{r}\right)\, dz \le \frac{\Delta_2(\V)}{2}\-int_{\Q_{\rho}} \V\left(\frac{|u-\ell|}{r}\right)\, dz + \frac{\Delta_2(\V)}{2}\-int_{\Q_{\rho}}\V\left(\frac{|\ell-\ell_{\rho}|}{r}\right)\, dz. 
\end{align*}
Hence, using \eqref{l-lrho}, the fact that $\V$ is increasing together with Jensen's inequality, we can infer that
\begin{align*}
\-int_{\Q_{r}} \V \left( \frac{|\ell- \ell_{\rho}|}{r}\right) \, dz \lesssim \-int_{\Q_{\rho}} \V \left( \frac{|u-\ell|}{r}\right)\, dz. 
\end{align*}
\end{proof}
\noindent An analogous reasoning leads to another basic inequality.
\begin{remark}
For an $N$-function $\V$ satisfying the $\Delta_2$-condition, we have
\begin{align*}
\-int_{\Q_{\rho}(z_{0})} \V \left( \frac{|u-(u)_{z_{0}, \rho}|}{r}\right)\,dz \leq \Delta_2(\V) \,\-int_{\Q_{\rho}(z_{0})} \V\left(\frac{|u-u_{0}|}{r}\right)\, dz,
\end{align*}
for any $u_0\in\R^N$ and for any $r>0$.
\end{remark}

\noindent Finally, we can show that $\ell_{z_0,\rho}$ is an almost minimizer of the functional $\displaystyle{\ell\mapsto \-int_{\Q_{\rho}(z_{0})} \V_{1+|D\ell|}\left(\frac{|u-\ell|}{\rho}\right)\, dz}$ amongst the affine functions $\ell: \R^{n}\ri \R^{N}$.

\begin{lem}\label{quasimin2}
Let $\V$ be an $N$-function satisfying $\Delta_2(\V,\V^*)<+\infty$, and let $u\in L^{\V}(\Q_{\rho}(z_{0}), \R^{N})$. There exists a constant $\kappa_1= \kappa_1(n,\Delta_2(\V,\V^*))>0$ such that
$$
 \-int_{\Q_{\rho}(z_{0})} \V_{1+|D\ell_{z_0,\rho}|}\left(\frac{|u-\ell_{z_0,\rho}|}{\rho}\right)\, dz\le \kappa_1  \-int_{\Q_{\rho}(z_{0})} \V_{1+|D\ell|}\left(\frac{|u-\ell|}{\rho}\right)\, dz,
 $$
 for every  affine function $\ell: \R^{n}\ri \R^{N}$.
\end{lem}

\begin{proof} 
From Lemma \ref{phia}, Lemma \ref{quasimin}, and Lemma \ref{shift}  we obtain 
\[
\begin{split}
 \-int_{\Q_{\rho}(z_{0})} \V_{1+|D\ell_{z_0,\rho}|}\left(\frac{|u-\ell_{z_0,\rho}|}{\rho}\right)\, dz&\le \kappa_0\,  \-int_{\Q_{\rho}(z_{0})} \V_{1+|D\ell_{z_0,\rho}|}\left(\frac{|u-\ell|}{\rho}\right)\, dz\\
  & \le c_\delta \-int_{\Q_{\rho}(z_{0})} \V_{1+|D\ell|}\left(\frac{|u-\ell|}{\rho}\right)\, dz +\delta\V_{1+|D\ell|}(|D\ell-D\ell_{z_0,\rho}|),\\
\end{split}
\]
using also the fact that $\V_{1+|a|}(|a-b|)\sim\V_{1+|b|}(|a-b|)$. Moreover, from \eqref{touse} we infer
\[
|D\ell_{z_0,\rho}-D\ell|\le (n+2)\-int_{\Q_\rho(z_0)}\frac{|u-\ell|}{\rho}\,dz.
\]
Inserting this above, applying the $\Delta_2$-condition, and  Jensen's inequality concludes the proof.
\end{proof}

\noindent In the same way you obtain the following fact.
\begin{remark}\label{mediamin2}
For an $N$-function $\V$ satisfying $\Delta_2(\V,\V^*)<\infty$, we have
\begin{align*}
\-int_{\Q_{\rho}(z_{0})} \V_{1+|(Du)_{z_0,\rho}|} \left( |Du-(Du)_{z_0,\rho}|\right)\,dz\leq \kappa_2 \,\-int_{\Q_{\rho}(z_{0})} \V_{1+|A|}\left(|Du-A|\right)\, dz,
\end{align*}
for any $A\in\R^{Nn}$, where $\kappa_2= \kappa_2(n,\Delta_2(\V,\V^*))>0$.
\end{remark}

\noindent
We conclude the section with an excess-decay-estimate for weak solutions to linear parabolic systems with constant coefficients \cite[Lemma 5.1]{Ca}.  This can be achieved along the lines of the classical proof with very minor changes, so we will consider only the main points of the proof referring for the rest to \cite{Ca}.

\begin{lem}[$\A$-Caloric $\psi$-Excess Estimate]\label{L:CaloricDecay}
Suppose that $h\in L^1(t_0-R^2,t_0;W^{1,1}(\B_R(x_0);\R^{N}))$ is $\A$-caloric, and let $\psi:[0,\infty)\to[0,\infty)$ be an increasing function. Then $h\in C^\infty(\Q_R(z_0);\R^{N})$ and the following excess estimate holds: for each $0<\rho<R$ and $0<\theta<1/4$, we have
\[
    \-int_{\Q_{\theta\rho}(z_0)}\psi\left(\left|
        \frac{\gamma(h-\ell^{(h)}_{z_0,\theta\rho})}{\theta\rho}\right| \right)dz
    \le
    \psi\left(C\theta\-int_{Q_\rho(z_0)}\left|
        \frac{\gamma(h-\ell^{(h)}_{z_0,\rho})}{\rho}\right|\,dz\right)
\]
where $\ell^{(h)}_r(x):=(h)_{z_0,r}+(Dh)_{z_0,r}(x-x_0)$ and  $C$ depends on only $n,N,L/\nu$.
\end{lem}

\begin{proof}
It is only necessary to prove the estimate, since the smoothness of $h$ is already contained in \cite{Ca}. As argued in \cite[Remark 3.2 and Lemma 3.3]{Bogelein}, 
we may use (5.9) and (5.12) in \cite{Ca} to show that there exists $C'=C'(n,L/\nu)<\infty$ such that
\begin{equation}\label{E:CaloricBnd}
    \sup_{\Q_{\rho/2}(z_0)} |D^2w|
    \le C'\-int_{\Q_\rho(z_0)}\left|\frac{w}{\rho^2}\right|\, d z
    \quad\text{ and }\quad
    \sup_{\Q_{\rho/2}(z_0)} |D^3w|
    \le C'\-int_{\Q_\rho(z_0)}\left|\frac{w}{\rho^3}\right|\, d z,
\end{equation}
for any $\A$-caloric map $w\in C^\infty(\Q_R(z_0);\R^{N})$. Define $w_r= h-\ell^{(h)}_{z_0,r}$. Then $w_r$ is $\A$-caloric and $w_r\in C^\infty(\Q_R(z_0);\R^N)$, for each $0<r<R$. Let $0<\theta<\frac{1}{4}$ and $0<\rho\le R$ be given. Using~\eqref{E:CaloricBnd}, the fact that $\Q_r(z_0)$ is a standard parabolic cylinder, and the fact that every derivative of $h$ is still $\A$-caloric, for each $(x,t)\in\Q_{\theta\rho}(z_0)$, we have
\begin{align*}
    |w_{\theta\rho}(x,t)|
    \le&\theta\rho\sup_{\Q_{\theta\rho}(z_0)}|Dh-(Dh)_{z_0,\theta\rho}|+\theta^2\rho^2\sup_{\Q_{\theta\rho}(z_0)}|\partial_th|\\
    \le&
    \theta^2\rho^2\left(\sup_{\Q_{\theta\rho}(z_0)}|D^2h|+\theta\rho\sup_{\Q_{\theta\rho}(z_0)}|\partial_t Dh|+\sup_{\Q_{\theta\rho}(z_0)}|\partial_t h|
        \right)\\
    \le&
    C''\theta^2\rho^2\left(\sup_{\Q_{\rho/2}(z_0)}
        |D^2w_\rho|+\theta\rho\sup_{\Q_{\rho/2}(z_0)}|D^3w_\rho|\right)\\
    \le&
    C\theta^2\-int_{\Q_{\rho}(z_0)}|w_\rho|\, d z.
\end{align*}
Here, $C''$ and $C$ depend on only $n$, $N$, and $L/\nu$. The result follows from this and the definition of $w_\rho$.
\end{proof}

\section{$\mathcal A$-caloric approximation}\label{Acal}

\noindent
To prove the partial regularity for non-degenerate parabolic systems with $\V$-growth, we shall compare the solution of our parabolic system with the solution of a linear parabolic system with constant coefficients. The comparison will be achieved by a generalization of the $\mathcal A$-caloric approximation lemma in Orlicz spaces. We emphasize that the approximation lemma requires no upper bound on the growth of $\V$.
\smallskip

\noindent
Recall that a function $f:[0,+\infty)\to[0,+\infty)$ is said to be almost increasing if there exists $\Lambda\ge 1$ such that $f(t)\le\Lambda f(s)$ for every $0\le t<s<+\infty$. We will consider the following assumptions for the $N$-function $\V$, more general with respect to Assumption~\ref{assumption}:
\begin{itemize}
    \item[(H1)] There exists a $\displaystyle{p_0>\frac{2n}{n+2}}$ such that $\displaystyle{\frac{\V(t)}{t^{p_0}}}$ is almost increasing, 
    \item[(H2)] $\V$ has a uniform doubling property near zero; i.e. $\displaystyle{\limsup_{t\to0^+} \frac{\V(2t)}{\V(t)}=\Delta_0(\V)<\infty}$.
\end{itemize}
In general, an $N$-function might not satisfy assumption (H2). For example, with
\[
    \ell_k(t)=\frac{1}{k!}+\frac{2^{k}(k-1)}{k!}(t-2^{-k}),
    \quad\text{ for }k\in\mathbb{N},
\]
the $N$-function
\[
    \V(t)=\left\{\begin{array}{ll}
        0, & t=0,\\
        \ell_{k+1}(t), & k\in\mathbb{N}\setminus \{1\}\text{ and }2^{-k-1}\le t<2^{-k}\\
        8t^2, & 2^{-2}\le t,
    \end{array}\right.
\]
is not uniformly doubling near zero since $(k+1)\ell_{k+1}(2^{-k-1})=\ell_k(2^{-k})$. For any $N$-function and $a>0$, however, (H2) is satisfied by the shifted function $\V_a(t)$. In fact,
\[
    \frac{\V_a(2t)}{\V_a(t)}\le \frac{4\V'(2a)}{\V'(a)},
    \quad\text{ for all }0<t\le \frac{a}{2}.
\]

\subsection{Additional Notation and Supporting Results}
For this section, we introduce some additional notation. There are also several supporting results used in the proof of the approximation lemma.

\noindent
First, we require a compactness principle of Simon. 
\begin{thm} \cite[Theorem 6]{Simon}\label{comp}
Suppose that $X\subseteq B\subseteq Y$ are Banach spaces with a compact embedding $X\to B$. Given $1<q\le \infty$, assume
\begin{itemize}
    \item $F$ is bounded in $L^q(0,T;B)\cap L^1_{loc}(0,T;X)$, 
    \item for all $0<t_1<t_2<T$, $\|f(\cdot+h)-f(\cdot)\|_{L^p(t_1,t_2;Y)}\to 0$ as $h\to0$, uniformly for $f\in F$. 
\end{itemize}
Then $F$ is relatively compact in $L^p(0,T;B)$ for all $1\le p<q$.
\end{thm}

\noindent
We also need to work with the Orlicz norm: given a measurable $E\subseteq\R^n$,
\[
    \|f\|_{L^*_{\V}(E)}=\sup\left\{\int_Ef(y)g(y) dy
    :\int_E\V^*(|g(y)|) dy\le 1\right\}.
\]
It can be verified~\cite{Krasnoselski} that the Orlicz space
\[
    L^*_{\V}(E)=\{f\in L^1(E): \|f\|_{L^*_{\V}(E)}<\infty\}
\]
is a Banach space. The Orlicz norm is equivalent to the Luxemborg norm (\cite{Krasnoselski}, p. 80). Moreover, as established in~\cite[Lemma 9.2 and p. 75]{Krasnoselski}, given $\{f_k\}_{k=1}^\infty\subseteq L^*_{\V}(E)$ and $f\in L^*_{\V}(E)$,
\begin{equation}\label{E:OConvMConv}
    \lim_{k\to\infty}\|f_k-f\|_{L^*_{\V}(E)}= 0
    \Longrightarrow
    \lim_{k\to\infty}\int_E\V(\lambda|f_k(x)-f(x)|)d x=0,
    \quad\text{ for all }\lambda>0.
\end{equation}

\noindent
We will also need the following
\begin{defn}
Given an open set $E\subseteq\R^n$ and $f\in L^1_{loc}(E)$, the \emph{(non-centered) Hardy-Littlewood maximal operator} is $M(f):E\to[0,\infty]$
\[
    Mf(x)=\sup_{\B\ni x}\-int_{\B\cap E}|f(y)|dy.
\]
Here the supremum is taken over all balls containing $x$.
\end{defn}

\noindent
It is well-known that the maximal operator is bounded on $L^p$, for $p>1$. From ~\cite[Corollary 4.3.3]{Harjulehto}, we have
\begin{cor}\label{C:MaxOpBnd}
Let an open set $E\subseteq\R^n$ and an $N$-function $\V$ be given. If $p>1$ and $\displaystyle{\frac{\V(t)}{t^{p}}}$ is almost increasing, then there exists a $\beta>0$ such that
\[
    \V(\beta Mf(x))^\frac{1}{p}\lesssim M\left(\V(f)^\frac{1}{p}\right)(x)
\]
for every ball $\B$, $x\in\B\cap E$, and $f\in L^{\V}(E)$ satisfying $\displaystyle{\int_E\V(f)dx\le1}$.
\end{cor}

\noindent
Finally, as explained in the proof of the $\mathcal A$-caloric excess estimate (Lemma~\ref{L:CaloricDecay}), we may use the regularity provided in (5.9) and (5.12) in~\cite{Ca} to show there is a $C=C(L/\nu)<\infty$ such that
\begin{equation}\label{E:SupCaloricBnd}
    \sup_{\Q_{\tau R}(z_0)}\left(|Dw|^2+|w|^2\right)
    \le C
    \-int_{\Q_R(z_0)}|w|^2dz,
\end{equation}
for any $\frac{3}{4}\le R\le 1$, $\frac{1}{2}\le\tau\le\frac{3}{4}$, and $\mathcal A$-caloric map $w\in L^1(t_0-R^2,0;W^{1,2}(\B_R(x_0);\R^N))$.

\subsection{The $\mathcal A$-Caloric Approximation Lemma}
With the preliminaries above, we can state and prove the main result for this section.

\begin{thm}\label{T:ApproxLemma}
Suppose that (H1) and (H2) are satisfied. Let $\e,\nu>0$ and $\nu<L<\infty$ be given. There exists $\delta_0\le 1$ and $1\le K_{\V}$ with the following property: for any $\gamma\in(0,1/\omega_n]$ ($\omega_n$ being the measure of the unit sphere in $\R^n$), any bilinear form $\mathcal A$ satisfying
\[
    \mathcal A(\xi,\xi)\ge\nu|\xi|^2
    \quad\text{ and }\quad
    |\mathcal A(\xi,\eta)|\le L|\xi||\eta|,
\]
and any approximately $\mathcal A$-caloric map $v\in L^\infty(t_0-\rho^2,t_0;L^2(\B_\rho(x_0);\R^N))\cap L^1(t_0-\rho^2,t_0;W^{1,1}(\B_\rho(x_0);\R^N))$ satisfying:
 \begin{itemize}
    \item $\V(|Dv|)\in L^1(t_0-\rho^2,t_0;L^1(\B_\rho(x_0)))$, 
    \item for some $0<\delta\le\delta_0$, 
\begin{equation}\label{E:ACaloricBnd}
    \left|\-int_{\Q_\rho(z_0)}\left( v\cdot\partial_t\eta-\mathcal A(Dv,D\eta)\right) dz\right|
    \le
    \delta\sup_{\Q_\rho(z_0)}|D\eta|,\:\text{ for all }\eta\in C^\infty_c(\Q_\rho(z_0);\R^N), 
\end{equation}
    \item and 
\begin{equation}\label{E:NormBnd}
    \sup_{t_0-\rho^2<t<t_0}\-int_{B_\rho}\left|\frac{v}{\rho}\right|^2d x
        +\-int_{\Q_\rho(z_0)}\left(\V\left(\left|\frac{v}{\rho}\right|\right)
        +\V(|Dv|)\right)  d z\le\gamma^2, 
\end{equation}
\end{itemize}
then there exists an $\mathcal A$-caloric map $h\in L^2(t_0-\rho^2/4,t_0;W^{1,2}(\B_{\rho/2}(x_0);\R^N))$ such that
\[
    \-int_{\Q_{\rho/2(z_0)}}\left(\left|\frac{\gamma h}{\rho/2}\right|^2
    +\V\left(\left|\frac{\gamma h}{\rho/2}\right|\right)
    +\V(|\gamma Dh|)\right) d z
    \le
    2^{n+2}K_{\V}\gamma^2,
\]
and
\[
    \-int_{\Q_{\rho/2(z_0)}}\left(\left|\frac{v-\gamma h}{\rho/2}\right|^2
    +\V\left(\left|\frac{v-\gamma h}{\rho/2}\right|\right)\right) d z
    \le\e\gamma^2.
\]
The constant $K_{\V}$ is defined in~\eqref{E:Kphi} and depends only on $\V$ and $C(L/\nu)$ in~\eqref{E:SupCaloricBnd}.
\end{thm}

\begin{proof}
We translate and rescale to $z_0=(0,0)$ and $\rho=1$. Assuming the alternative, there exists an $\e_0>0$, sequences $\gamma_k\in(0,1/\omega_n]$, bilinear forms $\mathcal A_k$, and maps $v_k\in L^\infty(-1,0;L^2(\B_1;\R^N))\cap L^1(-1,0;W^{1,1}(\B_1;\R^N))$, such that for each $k\in\mathbb{N}$, the following holds:
\begin{itemize}
\item[(i)] $\V(|Dv_k|)\in L^1(-1,0;L^1(\B_1))$, 
\item[(ii)] $\displaystyle{\-int_{\Q_1}\left( v_k\cdot\partial_t\eta-\mathcal A_k(Dv_k,D\eta)\right) d z
    \le
    \frac{1}{k}\displaystyle{\sup_{\Q_1}}|D\eta|,\:}$ for all $\eta\in C^\infty_0(\Q_1;\R^N)$,
\item[(iii)] $\displaystyle{\sup_{t\in(-1,0)}\-int_{\B_1}|v_k|^2 d x
+\-int_{\Q_1}\left(\V(|v_k|)+\V(|Dv_k|)\right) d z\le\gamma_k^2}$,
\item[(iv)] for any $\mathcal A_k$-caloric map $h\in C^\infty(\Q_{1/2};\R^N)$ satisfying
\[
    \-int_{\Q_{1/2}}\left(4|\gamma_k h|^2+\V(2|\gamma_kh|)
    +\V(|\gamma_kDh|)\right) d z
    \le
    2^{n+2}K_{\V}\gamma_k^2,
\]
we find
\[
    \-int_{\Q_{1/2}}\left(4|v_k-\gamma_kh|^2+\V(2|v_k-\gamma_kh|)\right) d z
    >\e_0\gamma_k^2.
\]
\end{itemize}

\noindent
By (iii) the sequence  $\{\V(|Dv_k|)\}_{k=1}^\infty$ is bounded in $L^1(-1,0;L^1(\B_1;\R^N))$. Assumption (H1) implies the existence of a constant $C$ such that
\[
    \|v_k\|_{L^{p_0}(\Q_1)}
    +\|Dv_k\|_{L^{p_0}(\Q_1)}\le C.
\]
It follows that, for a non-relabeled sequence, there exists a bilinear form $\mathcal A$ and a map $v\in L^2(\Q_1)\cap L^{p_0}(\Q_1)$ such that $Dv\in L^{p_0}(\Q_1)$ and
\[
    \left\{\begin{array}{ll}
         v_k\rightharpoonup v & \text{ in }L^2(\Q_1;\R^N),\\
         Dv_k\rightharpoonup Dv & \text{ in }L^{p_0}(\Q_1;\R^{Nn}),\\
         \mathcal A_k\to \mathcal A & \text{ in bilinear forms on $\R^{Nn}$},\\
         \gamma_k\to\gamma\in[0,1/\omega_n].
    \end{array}\right.
\]
The convexity of ${\V}$ 
implies
\begin{equation}\label{E:vBnd}
    \int_{\Q_1}\left(|v|^2+{\V}(|v|)+{\V}(|Dv|)\right) d z\le\gamma^2.
\end{equation}
Moreover, with the same argument used in~\cite{scheven}, we conclude that $v$ is $\mathcal A$-caloric and $v\in C^\infty(\Q_1;\R^N)$. From~\eqref{E:SupCaloricBnd},
\begin{equation}\label{E:vSupBnd}
    \sup_{\Q_{3/4}}{\left(|v|^2+|Dv|^2\right)}
    \le C(L/\nu)\-int_{\Q_1}|v|^2 dz
    \le C(L/\nu).
\end{equation}

\noindent As demonstrated in~\cite{scheven}, given $\ell>\frac{n+2}{2}$, for any $-1<t_1<t_2<-h$,
\begin{align*}
&\|v_k(\cdot,s)-v_k(\cdot,s+h)\|_{W^{-\ell,2}(\B_1)} \leq C\left( h^\frac{p_0-1}{p_0}+\frac{1}{k}\right), \quad \mbox{ for } s\in(t_1,t_2) \nonumber \\ 
& \Longrightarrow \int_{t_1}^{t_2}\|v_k(\cdot,s)-v_k(\cdot,s+h)\|^p_{W^{-\ell,2}(\B_1)} ds \leq C\left( h^\frac{p(p_0-1)}{p_0}+\frac{1}{k^p}\right), 
\end{align*}
for any $p\ge 1$.

\noindent Moreover, the sequence  $\{v_k\}_{k=1}^\infty$ is uniformly bounded in $L^\infty(-1,0;L^2(\B_1;\R^N))\cap L^1_{loc}(-1,0;W^{1,p_0}(\B_1;\R^N))$. With $X=W^{1,p_0}(\B_1;\R^N)$, $B=L^2(\B_1;\R^N)$, and $Y=W^{-\ell,2}(\B_1;\R^N)$, Theorem \ref{comp} yields the strong convergence (for a non-relabeled subsequence)
\begin{equation}\label{E:vkConvL2}
    v_k\to v\quad\text{ in }L^p(-1,0;L^2(\B_1;\R^N)),
\end{equation}
for any $p\ge 1$. This entails also that $ v_k\to v$ in $L^p(-1,0;L^1(\B_1;\R^N))$ for any $p\ge 1$, so $\|v_k-v\|_{L^1(B_1;\R^n)}\to 0$ for almost every $-1<t<0$.
\smallskip

\textbf{Claim}: $\displaystyle{\lim_{k\to\infty}\int_{\Q_{3/4}}{\V}(\lambda|v_k-v|) \,d z=0}$, for all $\lambda>0$.\\
First, we observe that $v_k(\cdot,t)$ and $v(\cdot,t)$  belong to $W^{1,1}(\B_1,\R^N)$, for almost every time $t\in(-1,0)$. We may therefore extend them to the whole of $\R^n$ in such a way that their extensions $\widetilde{v}_k(\cdot,t)$ and $\widetilde{v}(\cdot,t)$ belong to $W^{1,1}(\R^n,\R^N)$, and 
$$
\|\widetilde{v}_k(\cdot,t)-\widetilde{v}(\cdot,t)\|_{W^{1,1}(\R^n,\R^N)}\le C\|{v_k}(\cdot,t)-v(\cdot,t)\|_{W^{1,1}(\B_1,\R^N)}, \quad \mbox{ for a.e. } t\in (-1, 0), 
$$
where the constant $C$ depends only on $\B_1$. Let $\sigma$ be the standard mollifier and $\sigma_{\e}(x)=\frac{1}{\e^n}\sigma\left(\frac{|x|}{\e}\right)$.

\noindent
On account of~\eqref{E:OConvMConv}, it is enough to verify
\[
    \lim_{k\to\infty}\int_{\Q_{3/4}}|v_k-v|g dz=0
    \quad\text{ for all }g\in L^1(\Q_{3/4})\text{ such that }\|\V^*(|g|)\|_{L^1(\Q_{3/4})}\le 1.
\]
Fix $g\in L^1(\Q_{3/4})$ satisfying $\|\V^*(|g|)\|_{L^1(\Q_{3/4})}\le 1$. With $0<\e<1$ given, for each $-1<t<0$, we let $\widetilde{v}_k*\sigma_{\e}$ and $\widetilde{v}*\sigma_{\e}$ denote the mollifications of the extended maps $\widetilde{v}_k$ and $\widetilde{v}$ in the spatial direction. We have
\begin{align}\label{E:OrliczNormBnd}
&\left|\int_{\Q_{3/4}}|v_k(z)-v(z)|g(z)dz\right| \nonumber\\
&\leq \int_{\Q_{3/4}}|\widetilde{v}_k(x,t)-\widetilde{v}(x,t)||g(x,t)|dxdt \nonumber \\ 
&\leq \left[\int_{\Q_{3/4}}|(\widetilde{v}_k*\sigma_{\e})(x,t)-\widetilde{v}_k(x,t)||g(x,t)|dxdt +\int_{\Q_{3/4}}|(\widetilde{v}_k
            -\widetilde{v})*\sigma_{\e}(x,t)||g(x,t)|dxdt\right. \nonumber\\
&\qquad\left. +\int_{\Q_{3/4}}|(\widetilde{v}*\sigma_{\e})(x,t)-v(x,t)||g(x,t)|dxdt\right] \nonumber\\
&=: [I_{1,k}+I_{2,k}+I_{3,k}].        
\end{align}


\noindent
First, we examine $\displaystyle{\lim_{k\to\infty}I_{2,k}}$. Given $-1<t<0$, we use Young's convolution inequality to write 
\begin{align*}
    \sup_{k\in\mathbb{N}}
        \|(\widetilde{v}_k-\widetilde{v})*\sigma_{\e}(\cdot,t)\|_{L^\infty(\B_{3/4})}
    \le&
    \sup_{k\in\mathbb{N}}
        \|\widetilde{v}_k(\cdot,t)-\widetilde{v}(\cdot,t)\|_{L^1(\B_{3/4};\R^N)}
        \|\sigma_{\e}\|_{L^\infty(3/4)}\\
    \le&
    \sup_{k\in\mathbb{N}}
        \left(\|\widetilde{v}_k(\cdot,t)\|_{L^1(\B_{3/4};\R^N)}
            +\|\widetilde{v}(\cdot,t)\|_{L^1(\B_{3/4};\R^N)}\right)
        \|\sigma_{\e}\|_{L^\infty(3/4)}\\
    \le&
    C\sup_{k\in\mathbb{N}}
        \left(\|v_k(\cdot,t)\|_{L^2(\B_1;\R^N)}
            +\|v(\cdot,t)\|_{L^2(\B_1;\R^N)}\right)\|\sigma_{\e}\|_{L^\infty(\R^n)}\\
    \le&
    C\|\sigma_{\e}\|_{L^\infty(\R^n)}.
\end{align*}
For the last two inequalities, we used (iii).
Since $g\in L^1(\Q_{3/4})$ and
\[
    \lim_{k\to\infty}|(\widetilde{v}_k-\widetilde{v})*\sigma_{\e}(x,t)||g(x,t)|\to0
    \quad\text{ for a.e. }(x,t)\in\Q_{3/4},
\]
we may use the dominated convergence theorem to conclude that $\displaystyle{\lim_{k\to\infty}I_{2,k}=0}$.

\noindent
Now, we turn to bounding $I_{1,k}$ and $I_{3,k}$. The arguments for each term are similar, so we focus on $I_{1,k}$. We will use some results contained in \cite{Harjulehto}. As provided in \cite{Harjulehto} (p. 135), the following pointwise estimate holds for functions in $W^{1,1}_{loc}(\R^n,\R^N)$:
\[
|\widetilde{v}_k(x,\cdot)-(\widetilde{v}_k\ast\sigma_{\e})(x,\cdot)|\le \e\int_0^1(|D \widetilde{v}_k|\ast\sigma_{\e\tau})(x,\cdot)\,d\tau,
\]
almost everywhere in $\B_{3/4}$. Let $\beta>0$ be the constant from Corollary~\ref{C:MaxOpBnd}. Combining the bound above with \cite[Lemma 4.4.6]{Harjulehto}, for each $-1<t<0$, we obtain
\[
|\widetilde{v}_k(x,t)-(\widetilde{v}_k\ast\sigma_{\e})(x,t)|\le\frac{2\e}{\beta} M(\beta|D \widetilde{v}_k|(\cdot,t))(x),
\quad\text{ for a.e. }x\in\B_{3/4}.
\]
Here $M$ is the (non-centered Hardy-Littlewood) maximal function defined earlier. Observing that, for almost every $t\in(-1,0)$, we have
\begin{align*}
M(\beta|D\widetilde{v}_k(\cdot,t)|)(x)&=\sup_{\B\ni x}\frac{\beta}{|\B|}\int_{\B\cap\B_1} |D\widetilde{v}_k(y,t)|\,dy\\
&\le \sup_{\Q\ni (x,t)}\frac{\beta}{|\Q|}\int_{\Q\cap\Q_1} |D\widetilde{v}_k(y,s)|\,dy\,ds=M(\beta|D\widetilde{v}_k|)(x,t),
\end{align*}
we deduce that
\[
    |\widetilde{v}_k(z)-(\widetilde{v}_k\ast\sigma_{\e})(z)|
    \le\frac{2\e}{\beta} M(\beta|D\widetilde{v}_k|)(z),
    \quad\text{ for a.e. }z\in\Q_{3/4}.
\]
Incorporating this into the definition of $I_{1,k}$ and applying Young's inequality, we may write
\begin{align*}
    I_{1,k}
    \le&
    \frac{2\e}{\beta}\int_{Q_{3/4}}M(\beta|D\widetilde{v}_k|)(z)|g(z)|dz\\
    \le&
    \frac{2\e}{\beta}\left[\int_{Q_{3/4}}\V(M(\beta |D\widetilde{v}_k|)(z))dz
    +\int_{Q_{3/4}}\V^*(|g(z)|)dz\right]\\
    \le&
    \frac{2\e}{\beta}\left[\int_{Q_{3/4}}\V(M(\beta |D\widetilde{v}_k|)(z))dz
    +1\right],
\end{align*}
where we have used $\|\V^*(|g|)\|_{L^1(Q_{3/4})}\le 1$ in the last inequality. Recalling that $\displaystyle{\int_{\Q_1}\V(|D v_k|)\,dz\le 1}$ and that $\displaystyle{\frac{\V(t)}{t^{p_0}}}$ is almost increasing, we can use Corollary~\ref{C:MaxOpBnd} to infer
$$
\V(\beta M(|D v_k|)(z))^\frac{1}{p_0}
\lesssim M\left(\V(|Dv_k|)^{\frac{1}{p_0}}\right)(z),
$$
almost everywhere in $\Q_1$. Thus
\[
\int_{\Q_{3/4}}\V(\beta M(|D v_k|))\,dz\lesssim \int_{\Q_1}\left(M\left(\V(|D v_k|)^{\frac{1}{p_0}}\right)\right)^{p_0}dz\lesssim \int_{\Q_1}\V(|D v_k|)\,dz\le 1,
\]
since $M$ is bounded in $L^{p_0}$. We conclude that $I_{1,k}\lesssim\e/\beta$. A similar argument shows $I_{3,k}\lesssim\e\beta$, as well.

\noindent
Returning to~\eqref{E:OrliczNormBnd}, we have shown
\[
    \lim_{k\to\infty}\left|\int_{Q_{3/4}}|v_k(z)-v(z)|g(z)dz\right|\lesssim\frac{\e}{\beta}.
\]
Here $\beta$ is independent of $k$. Since $\e>0$ and $g\in L^1(Q_{3/4})$ satisfying $\|\V^*(|g|)\|_{L^1(Q_{3/4})}\le 1$ were both arbitrary, the claim is proved.

\noindent
Next, we produce a sequence $\{h_k\}_{k=1}^{\infty}\in C^\infty(\Q_{1/2};\R^N)$ of $\mathcal A$-caloric maps that will contradict (iv) for $k$ sufficiently large. 

\underline{Case 1: $\gamma_k\to0$}: In this case, clearly $v=0$, $v_k\to0$ strongly in $L^2(\Q_1)$ and $\displaystyle{\int_{\Q_1}\V(2|v_k|) d z\to 0}$. Thus, we obtain a contradiction to (iv) with $h\equiv0$.\\

\underline{Case 2: $\gamma_k\to\gamma\in(0,1/\omega_n]$}:  For each $k\in\mathbb{N}$, let $h_k$ be the unique solution to
\[
    \left\{\begin{array}{l}
        \displaystyle{\int_{\Q_{3/4}}\left( h_k\cdot\partial_t\eta-\mathcal A_k(Dh_k,D\eta)\right) d z=0}
            \quad\text{ for all }\eta\in C^\infty_c(\Q_{3/4};\R^N)\\
        h_k=\gamma_k^{-1}v\quad\text{ on }\partial_{\mathcal P}\Q_{3/4}.
    \end{array}\right.
\]
Since $v\in C^\infty(\overline{\Q}_{3/4};\R^N)$ and $\gamma_k\to\gamma>0$, so is each $h_k$. As shown in \cite{BDM} and \cite{scheven}, we have
\[
    \lim_{k\to\infty}\int_{\Q_{3/4}}
        \left(|v-\gamma_k h_k|^2+|Dv-\gamma_k Dh_k|^2\right) d z=0.
\]
Thus
\begin{equation}\label{E:hkConvL2}
    \gamma_k h_k\to v\text{ and }\gamma_k Dh_k\to Dv
    \quad\text{ in }L^2(\Q_{3/4}).
\end{equation}
This implies convergence in measure for both sequences. Moreover, after taking a non-relabeled subsequence if necessary, the bounds in~\eqref{E:SupCaloricBnd} and~~\eqref{E:vSupBnd} imply
\[
    \sup_{\Q_{1/2}}\left(|\gamma_kh_k|+|\gamma_kDh_k|\right)\le 2^{n+1}C(L/\nu).
\]
Thus, since $\sup_{k\in\mathbb{N}}\|\gamma_kh_k-v\|_{L^\infty(Q_{1/2})}+\|\gamma_kDh_k-Dv\|_{L^\infty(Q_{1/2})}\le 2^{n+3}C(L/\nu)$,
\begin{equation}\label{E:hkConvOrlicz}
    \lim_{k\to\infty}\int_{\Q_{1/2}}
    \left({\V}(\lambda|\gamma_k h_k-v|)+{\V}(\lambda|\gamma_k Dh_k-Dv|)\right) d z=0,
    \quad \text{ for all }\lambda>0.
\end{equation}
To finish the proof, define
\begin{equation}\label{E:Kphi}
    K_{\V}=8+\sup\left\{\frac{\V(4t)}{\V(t)}:0<t\le 2^{n+2}C(L/\nu)\right\}.
\end{equation}
Note that $K_{\V}$ must be finite due to assumption (H2). Using ~\eqref{E:vBnd} and the convexity of $\V$,
\begin{align*}
\lim_{k\to\infty} &\-int_{\Q_{1/2}}\left(4|\gamma_kh_k|^2+{\V}(2|\gamma_kh_k|)+{\V_1}(|\gamma_kDh_k|)\right) dz\\
&\leq \frac{1}{2}\lim_{k\to\infty}\-int_{\Q_{1/2}} \left(16|\gamma_kh_k-v|^2+{\V}(4|\gamma_kh_k-v|)+{\V}(2|\gamma_kDh_k-Dv|)\right) d z\\
&\quad+\frac{1}{2}\-int_{\Q_{1/2}}\left(16|v|^2+{\V}(4|v|)+{\V}(2|Dv|)\right) d z\\
&\leq 2^{n+1}K_{\V}\-int_{\Q_1} \left(|v|^2+{\V}(|v|)+{\V}(|Dv|)\right) d z\\
&\leq 2^{n+1}K_{\V}\gamma^2.
\end{align*}
Since $\gamma>0$, for $k$ sufficiently large, we have $\gamma<2\gamma_k$, and the $\mathcal A_k$-caloric map $h_k$ satisfies
\[
    \-int_{\Q_{1/2}}
        \left(4|\gamma_kh_k|^2+{\V_1}(2|\gamma_kh_k|)+{\V}(|\gamma_kDh_k|)\right) d z
    <
    2^{n+2}K_{\V}\gamma_k^2.
\]
Similarly, using the convergences provided by the claim,~\eqref{E:vkConvL2},~\eqref{E:hkConvL2}, and~\eqref{E:hkConvOrlicz} we conclude that
\begin{align*}
    \lim_{k\to\infty}&
    \-int_{\Q_{1/2}}\left(4|v_k-\gamma_kh_k|^2+{\V}(2|v_k-\gamma_kh_k|)\right) d z=0,
\end{align*}
which provides the contradiction to (iv).
\end{proof}

\begin{remark}\label{R:ApproxLemma}
For the shifted function $\V_a$,
\[
    K_{\V_a}\le 8+\max\left\{\left(\frac{4\V'(2a)}{\V'(a)}\right)^2,
        \frac{\V_a(2^{n+4}C(L/\nu))}{\V_a(a/4)}\right\}.
\]
If the function $\V$ is doubling, then $K_{\V}=8+\Delta_2(\V)^2$.
\end{remark}


\section{Caccioppoli}\label{Caccioppoli}

\noindent
Let us prove the following Caccioppoli inequality for standard parabolic cylinders.

\begin{thm}\label{cacc2}
Let $u\in C^0(-T,0; L^2(\Omega,\R^N))\cap L^1(-T,0; W^{1,1}(\Omega,\R^N))$  be a weak solution to  \eqref{P} satisfying $\V(|Du|)\in L^1(-T,0; L^1(\Omega))$. Under hypotheses $(a_{1})$-$(a_{4})$ and Assumption \ref{assumption}, given a standard cylinder $\Q_{R}(z_0)\subset\Omega_T$, with center $z_0=(x_0,t_0)$, and any affine map $\ell: \R^{n}\to \R^{N}$ and $0<r<R$, we have 
\begin{equation*}
\begin{split}
\sup_{t_0-{r^2}<s<t_0}\ &\int_{\B_{r}(x_0)}\left |{u(x,s)-\ell(x)}\right|^2dx+\int_{\Q_{r}(z_0)}\V_{1+|D\ell|}(|Du-D\ell|)dz\\
 \leq c_0&\int_{\Q_{R}(z_0)}\left [\V_{1+|D\ell|}\left(\left|\frac{u-\ell}{R-r}\right|\right)+\left|\frac{u-\ell}{R-r}\right|^2\right]dz,
\end{split}
\end{equation*}
where $c_0$ depends only on $n, N, L,\nu, p_0,p_1$.
\end{thm}

\begin{proof}
For notational brevity, we put $M=1+|D\ell|$. Without loss of generality we may assume $z_0 =(0,0)$. For a generic radius $\rho$, we denote $\Q_\rho(z_0)=\Q_\rho$ and $\B_\rho(x_0)=\B_\rho$. 
Let us consider the function $\eta(x,t)=\chi^{p_1}(x)\zeta^2(t)(u(x,t)-\ell(x))$ as a test function in \eqref{P}, where $\chi$ is a standard cutoff function between $\B_{r}$ and $\B_R$, and $\zeta\in C^0(\R)$ is defined by 
\begin{align*}
\left\{
\begin{array}{ll}
\zeta(t)=0, & t\in (-\infty,-{R^2})\\
\zeta_t(t)=\frac{1}{R^2-r^2}, &t\in (-{R^2},-{r^2})\\
\zeta(t)=1, &t\in(-{r^2},s)\\
\zeta_t(t)=-\frac{1}{\varepsilon},& t\in(s, s+\varepsilon)\\
\zeta(t)=0, &t\in (s+\varepsilon,+\infty)
\end{array}
\right.
\end{align*}
for $-r^{2}<s<0$ and $0< \e \leq |s|$. 
We have
$$
\int_{\Q_R}\chi^{p_1}\zeta^2 a(Du)\cdot(Du-D\ell)\,dz=-p_1\int_{\Q_R}\chi^{p_1-1}\zeta^2a(Du)\cdot\left[D\chi\otimes(u-\ell)\right]\,dz+\int_{\Q_R}u\cdot\eta_t\,dz.
$$
Noting that $\displaystyle{\int_{\Q_R}a(D\ell)\cdot D\eta \,dz =0}$ and $\displaystyle{\int_{\Q_R}\ell\cdot\eta_t \,dz=0}$, we obtain
\begin{align}\label{E:I-IIIDef}\begin{split}
I:=&\int_{\Q_R}\chi^{p_1}\zeta^2(a(Du)-a(D\ell))\cdot(Du-D\ell)\,dz\\
=&-p_1\int_{\Q_R}\chi^{p_1-1}\zeta^2(a(Du)-a(D\ell))\cdot\left[D\chi\otimes(u-\ell)\right]\,dz+\int_{\Q_R}(u-\ell)\cdot\eta_t\,dz=:I\!I+I\!I\!I
\end{split}
\end{align}
The left hand side can be estimated thanks to Remark \ref{rem1}, leading to
\begin{equation}\label{E:CaccIBnd}
I\ge c\int_{\Q_R}\chi^{p_1}\zeta^2\V_{M}(|Du-D\ell|)\,dz
\end{equation}
and
\begin{equation*}
|I\!I|\le c\int_{\Q_R}\chi^{p_1-1}\zeta^2\V'_{M}(|Du-D\ell|)|D\chi||u-\ell|\,dz.
\end{equation*}
Using Young's inequality and \eqref{t1conj} together with \eqref{t3}, we derive the following bound for $I\!I$:
\begin{equation}\label{E:CaccIIBnd}
\begin{split}
|I\!I|&\le c\, \delta\int_{\Q_R}\zeta^2\V_{M}^*\left(\V'_{M}(|Du-D\ell|)\chi^{p_1-1}\right)\,dz+c(\delta)\int_{\Q_R}\zeta^2\V_{M}(|u-\ell||D\chi|)\,dz\\
&\le c\,\delta\int_{\Q_R}\chi^{p_1}\zeta^2\V_{M}(|Du-D\ell|)\,dz+c(\delta)\int_{\Q_R}\zeta^2\V_{M}\left(\frac{|u-\ell|}{R-r}\right)\,dz.
\end{split}
\end{equation}
Choosing $\delta$ sufficiently small, we can absorb the first integral of the right hand side into the left. Finally, expanding the derivative $\eta_t$, we may write (recalling Remark \ref{ut})
\begin{equation*}
\begin{split}
I\!I\!I&=2\int_{\Q_R}\chi^{p_1}\zeta\zeta_t|u-\ell|^2dz+\int_{\Q_R}\chi^{p_1}\zeta^2(u-\ell)\cdot(u-\ell)_t\,dz\\
&=2\int_{\Q_R}\chi^{p_1}\zeta\zeta_t|u-\ell|^2dz+\frac{1}{2}\int_{\Q_R}\chi^{p_1}\zeta^2\frac{\partial}{\partial t}|u-\ell|^2\,dz.
\end{split}
\end{equation*}
So, an integration by parts yields
$$
I\!I\!I=\int_{\Q_R}\chi^{p_1}\zeta\zeta_t|u-\ell|^2dz.
$$
Exploiting the definition of $\zeta$, we obtain
\begin{equation*}
\begin{split}
I\!I\!I&=\frac{1}{R^2-r^2}\int_{-{R^2}}^{-{r^2}}\int_{\B_R}|u-\ell|^2\chi^{p_1}\zeta\, dx\,dt-\frac{1}{\varepsilon}\int_s^{s+\varepsilon}\int_{\B_R}|u-\ell|^2\chi^{p_1}\zeta\,dx\,dt\\
&\le \int_{\Q_R}\left|\frac{u-\ell}{R-r}\right|^2\,dz-\frac{1}{\varepsilon}\int_s^{s+\varepsilon}\int_{\B_R}|u-\ell|^2\chi^{p_1}\zeta\,dx\,dt,
\end{split}
\end{equation*}
since $R^2-r^2\ge (R-r)^2$. Incorporating the above bound and the bounds for $I$ and $I\!I$, in~\eqref{E:CaccIBnd} and~\eqref{E:CaccIIBnd}, into~\eqref{E:I-IIIDef}, we deduce that
\begin{equation*}
\begin{split}
&\frac{1}{\varepsilon}\int_s^{s+\varepsilon}\int_{\B_R}|u-\ell|^2\chi^{p_1}\zeta\,dx\,dt+\int_{\Q_R}\chi^{p_1}\zeta^2\V_{M}(|Du-D\ell|)\,dz\le c\,\int_{\Q_R}\left|\frac{u-\ell}{R-r}\right|^2\,dz+c\,\int_{\Q_R}\V_{M}\left(\frac{|u-\ell|}{R-r}\right)\,dz.
\end{split}
\end{equation*}
Recalling the definition of $\zeta$ and $\chi$, we may take the limit as $\varepsilon\to 0$ to get
\begin{equation*}
\begin{split}
&\int_{\B_{r}}|u(s,x)-\ell(x)|^2dx+\int_{-{r^2}}^s\int_{\B_{r}}\V_{M}(|Du-D\ell|)\,dz\le c\,\int_{\Q_{R}}\left|\frac{u-\ell}{R-r}\right|^2\,dz+c\int_{\Q_R}\V_{M}\left(\frac{|u-\ell|}{R-r}\right)\,dz.
\end{split}
\end{equation*}
We use the previous inequality twice: firstly, by dropping the second term in the left hand side, and taking the supremum over $s\in(-{r^2},0)$; secondly, by dropping the first term in the left hand side and letting $s$ tend to $0$. By summing up the two resulting contributions, this gives the result.
\end{proof}

\noindent
Finally, an application of the Caccioppoli inequality in Theorem~\ref{cacc2} and (\ref{E:ShiftDouble}) produces the following

\begin{cor}\label{cacc3}
Let $u\in C^0(-T,0; L^2(\Omega,\R^N))\cap L^1(-T,0; W^{1,1}(\Omega,\R^N))$  be a weak solution to  \eqref{P} satisfying $\V(|Du|)\in L^1(-T,0; L^1(\Omega))$. Under hypotheses $(a_{1})$-$(a_{4})$ and Assumption \ref{assumption}, given any standard parabolic cylinder $\Q_{\rho}(z_0)\subset\Omega_T$, with center in $z_0=(x_0,t_0)$, and any affine map $\ell: \R^{n}\to \R^{N}$, we have 
\begin{equation*}
\begin{split}
\sup_{t_0-(\frac{\rho}{2})^2<s<t_0}\ &\-int_{\B_{\frac{\rho}{2}}(x_0)}\left |\frac{u(x,s)-\ell(x)}{{\rho}}\right|^2dx+\-int_{\Q_{\frac{\rho}{2}}(z_0)}\V_{1+|D\ell|}(|Du-D\ell|)dz\\
 \leq&
  c_02^{n+p_1+2}\-int_{\Q_{\rho}(z_0)}\left [\V_{1+|D\ell|}\left(\left|\frac{u-\ell}{\rho}\right|\right)+\left|\frac{u-\ell}{\rho}\right|^2\right]dz,
\end{split}
\end{equation*}
where $c_0$ depends only on $n, N, L,\nu, p_0,p_1$.
\end{cor}


\section{Poincar\'e type inequalities}\label{poinc}

\noindent
We begin this section  providing a Poincar\'e type inequality valid for solutions to certain parabolic-like systems. The proof follows the same lines as \cite[Lemma 3.1]{BDM}.
\begin{lem}\label{L:PoincareLemma}
Let $\psi$ be an $N$-function satisfying $\Delta_2(\psi,\psi^*)<\infty$. With $t_1<t_2$ and $U\subseteq\R^n$, suppose that $\xi\in L^1(U\times(t_1,t_2),\R^{Nn})$ and 
$w\in C^0(t_1,t_2, L^{2}(U, \R^{N}))\cap L^{1}(t_1,t_2, W^{1, 1}(U,\R^N))$ satisfy $\psi(|Dw|)\in L^{1}(t_1,t_2, L^{1}(U))$ and
\begin{equation}\label{equa}
\int_{U\times(t_1,t_2)} (w\cdot\zeta_t-\xi \cdot D\zeta)\,  dz=0,
\quad\text{ for any }\zeta\in C^\infty_c(U\times (t_1,t_2),\R^N).
\end{equation}
Then for any parabolic cylinder $\Q_\rho(z_0)\subset U\times (t_1,t_2)$, we have
$$
\-int_{\Q_\rho(z_0)}\psi\left(\left|\frac{w-(w)_{z_0,\rho}}{\rho}\right|\right)\, dz\le c_1\left[\-int_{\Q_\rho(z_0)}\psi(|Dw|) \,dz+\psi\left(\-int_{\Q_\rho(z_0)}|\xi|\,dz\right)\right]
$$
with $c_1$ depending on only $n,N,$ and $\Delta_2(\psi,\psi^*)$.
\end{lem}

\begin{proof}
We fix a nonnegative symmetric weight function $\eta\in C^\infty_c(\B_\rho(x_0))$ such that 
\begin{align}\label{asseta}
\eta \geq 0, \quad \-int_{\B_{\rho}(x_{0})} \eta \, dx=1 \quad \mbox{ and } \quad \|\eta\|_{\infty}+ \rho \|D\eta\|_{\infty} \leq c_{\eta}, 
\end{align}
and, for $t\in \left(t_0-{\rho^2},t_0 \right)$ we denote
$$
(w)_\eta(t)=\-int_{\B_\rho(x_0)}w(x,t)\eta(x)\,dx,
$$
as well as 
$$
(w)_\eta=\-int_{\Q_\rho(z_0)}w(x,t)\eta(x)\,dz.
$$
By the triangle inequality and the $\Delta_2$-condition we have
\begin{align*}
&\-int_{\Q_\rho(z_0)}\psi\left(\left|\frac{w-(w)_{z_0,\rho}}{\rho}\right|\right)dz\\
&\quad \le c\left[\-int_{\Q_\rho(z_0)}\psi\left(\left|\frac{w-(w)_\eta(t)}{\rho}\right|\right)dz +\-int_{t_0-{\rho^2}}^{t_0}\psi\left(\left|\frac{(w)_\eta(t)-(w)_\eta}{\rho}\right|\right)dt + \psi\left(\left|\frac{(w)_\eta-(w)_{z_0,\rho}}{\rho}\right|\right)\right]\\
&\quad =:c\,(I+I\!I+I\!I\!I),
\end{align*}
with the obvious meaning of $I$,$I\!I$, and $I\!I\!I$. Since $\Delta_2(\psi,\psi^*)<\infty$, we may bound $I$ by applying Poincar\'e's inequality for vanishing $\eta$-mean value (see \cite[Theorem 7]{DE}) slicewise with respect to $x$: for a.e. $t\in(t_0-{\rho^2},t_0)$,
\begin{equation*}
I\le c(n,\Delta_2(\psi,\psi^*))\-int_{\Q_\rho(z_0)}\psi(|Dw|)dz.
\end{equation*}
To bound $I\!I\!I$, we use Jensen's inequality followed by the triangle inequality and $\Delta_2$-condition to infer
\begin{equation*}
\begin{split}
I\!I\!I\le\-int_{\Q_\rho(z_0)}\psi\left(\left|\frac{w-(w)_\eta}{\rho}\right|\right)\,dz&\le c\,\left[\-int_{\Q_\rho(z_0)}\psi\left(\left|\frac{w-(w)_\eta(t)}{\rho}\right|\right)dz+\-int_{t_0-{\rho^2}}^{t_0}\psi\left(\left|\frac{(w)_\eta(t)-(w)_\eta}{\rho}\right|\right)dt\right]\\
&=c(I+I\!I).
\end{split}
\end{equation*}
So it remains to estimate $I\!I$. For this, we recall that $w$ is a weak solution of the parabolic system \eqref{equa}. Even if the solution $w$, need not be differentiable in the time variable, with Steklov averages, we may rewrite~\eqref{equa} as
\[
    \int_U\left(w_t\cdot\zeta+\xi\cdot D\zeta\right)\,dx=0
    \quad
    \text{ for all }\zeta\in C^\infty_c(U,\R^N)
    \text{ and for a.e. }t\in(t_0-\rho^2,t_0).
\]
For $i=1,\dots,N$, let $e_i\in\R^N$ denote the unit vector in the $i$-th coordinate direction. With $(t,\tau)\subset(t_0-{\rho^2},t_0)$, we use \eqref{equa} and \eqref{asseta}, with $\zeta=\eta e_i\in C^\infty_c(U;\R^N)$, to write 
\begin{align*}
|\left[(w)_\eta(t)-(w)_\eta(\tau)\right]\cdot e_i|
    &=\left|\int_t^\tau\partial_s\left[(w)_\eta(s)\right]
        \cdot e_i\,ds\right|
    =\left|\int_t^\tau\-int_{\B_\rho(x_0)}w_s
        \cdot(\eta e_i) \,dx\,ds\right|\\
    &=\left|\int_t^\tau\-int_{\B_\rho(x_0)}\xi
        \cdot (D\eta e_i)\,dx\,ds\right|
     \le \|D\eta\|_\infty\int_t^\tau\-int_{\B_\rho(x_0)}|\xi|dx\,ds\\
     &\le \frac{c}{\rho}\int_t^\tau\-int_{\B_\rho(x_0)}|\xi|dx\,ds\le c\,{\rho}\-int_{\Q_\rho(z_0)}|\xi|dz.
\end{align*}
Summing over each component, we conclude that
\begin{equation*}
I\!I\le c\,\psi\left(\-int_{\Q_\rho(z_0)}|\xi|dz\right).
\end{equation*}
Combining these estimates we obtain the desired Poincar\'e type inequality.
\end{proof}

\begin{remark}
Suppose that $\mathcal{A}$ is a bilinear form on $\R^{Nn}$ and there is $\Lambda<\infty$ such that $|\mathcal{A}(\eta_1,\eta_2)|\le\Lambda|\eta_1||\eta_2|$, for all $\eta_1,\eta_2\in\R^{Nn}$. If $h\in C^\infty(\Q_\rho(z_0);\R^N)$ is $\mathcal{A}$-caloric, then we may identify a $\xi\in C^\infty(\Q_\rho(z_0);\R^{Nn})$ so that, at each $z\in\Q_\rho(z_0)$, we have $\mathcal{A}(Dh(z),\eta)=\xi(z)\cdot\eta$, for all $\eta\in\R^{Nn}$. Then $|\xi|\le\Lambda|Dh|$ and Lemma~\ref{L:PoincareLemma} and Jensen's inequality imply
\begin{equation*}
\-int_{\Q_\rho(z_0)}\psi\left(\left|\frac{h-(h)_{z_0,\rho}}{\rho}\right|\right)\, dz\le c^*_1\-int_{\Q_\rho(z_0)}\psi(|Dh|) \,dz,
\end{equation*}
with $\psi$ an $N$-function satisfying $\Delta_2(\psi,\psi^*)<\infty$ and $c^*_1$ depending on $n,N,\Lambda,$ and $\Delta_2(\psi,\psi^*)$.
\end{remark}


\noindent
The following Poincar\'e type inequality for weak solutions of \eqref{P} is a consequence of the previous lemma.

\begin{thm}\label{poincare}
Under the assumptions $(a_{1})$-$(a_{4})$ and Assumption \ref{assumption}, suppose $u\in C^0(-T,0; L^2(\Omega,\R^N))\cap L^1(-T,0; W^{1,1}(\Omega,\R^N))$ is a weak solution to \eqref{P} such that $\V(|Du|)\in L^1(-T,0;L^1(\Omega))$. Let $\Q_\rho(z_0)\subset\Omega_T$ be a standard parabolic cylinder. Then, for any $N$-function $\psi$ satisfying $\Delta_2(\psi,\psi^*)<\infty$ and any $A\in\R^{Nn}$, we have
\begin{align*}
\-int_{\Q_\rho(z_0)}\psi\left(\left|\frac{u-(u)_{z_0,\rho}-A(x-x_0)}{\rho}\right|\right)dz \le c_2\left[\-int_{\Q_\rho(z_0)}\psi(|Du-A|)\, dz +\psi\left(\-int_{\Q_\rho(z_0)}\V'_{1+|A|}(|Du-A|)\right) \, dz \right]
\end{align*}
with $c_2$ depending on $n,N,p_0,p_1,\nu,L,$ and $\Delta_2(\psi,\psi^*)$.
\end{thm}

\begin{proof}
Without loss of generality we can assume that $z_0=(0,0)$. 
Exploiting ~\eqref{P} and using the fact that $\displaystyle{\int_{\Omega_T}Ax\cdot \zeta_t\,dz=0}$ and $\displaystyle{\int_{\Omega_T}a(A)\cdot D\zeta\,dz = 0}$ for any function $\zeta\in C^\infty_c(\Omega_T , \R^N)$, we have 
$$
\int_{\Omega_T} \left[(u-Ax)\cdot\zeta_t-(a(Du)-a(A))\cdot D\zeta \right]\,dz=0.
$$
Therefore, we can apply the Lemma~\ref{L:PoincareLemma} with $w=u-Ax$ and $\xi=a(Du)-a(A)$. As $Ax$ has zero-mean on $\Q_\rho$, we obtain:
\begin{align*}
\-int_{\Q_\rho(z_0)}\psi\left(\left|\frac{u-(u)_\rho-Ax}{\rho}\right|\right)dz\le c_1\left[\-int_{\Q_\rho(z_0)}\psi(|Du-A|)dz+\psi\left(\-int_{\Q_\rho(z_0)}|a(Du)-a(A)|\,dz\right)\right].
\end{align*}
By means of Remark \ref{rem1}, we can estimate the right-hand side of the above inequality, and we get
\begin{align*}
\-int_{\Q_\rho(z_0)}\psi\left(\left|\frac{u-(u)_\rho-Ax}{\rho}\right|\right)dz\le c_2&\left[\-int_{\Q_\rho(z_0)}\psi(|Du-A|)dz+\psi\left(\-int_{\Q_\rho(z_0)}\V'_{1+|A|}(|Du-A|)\,dz\right)\right].
\end{align*}
\end{proof}

\noindent
We conclude this section by proving a weird Sobolev-Poincar\'e inequality for solutions of \eqref{P}. The proof follows the same lines as \cite[Lemma 3.4]{ho}. A special application of the inequality is required for Theorem~\ref{sp2}, which we ultimately use to establish the main regularity result in Theorem~\ref{main}.

\begin{thm}\label{sp}
Under the assumptions $(a_{1})$-$(a_{4})$ and Assumption \ref{assumption}, suppose $u\in C^0(-T,0; L^2(\Omega,\R^N))\cap L^1(-T,0; W^{1,1}(\Omega,\R^N))$ is a weak solution to \eqref{P} such that $\V(|Du|)\in L^1(-T,0;L^1(\Omega))$. Let $\Q_{\rho}(z_0)\subseteq\Omega_T$ be a standard parabolic cylinder and $\psi$ be an $N$-function satisfying Assumption \ref{assumption} (for some exponents $1<q_0\le q_1$). Then, for any $A\in\R^{Nn}$, any $\theta_0>0$ satisfying
\begin{equation}\label{E:theta0}
\theta_0q_0 \in (1,n)\qquad\text{ and }\qquad\frac{nq_1}{nq_1+2q_0}\le\theta_0\le 1,
\end{equation}
and each $\frac{\rho}{2}\le r<R\le \rho$, we have
\begin{equation}\label{E:sp}
\begin{split}
&\-int_{\Q_r(z_0)}\psi\left(\left|\frac{u-(u)_{z_0,r}-A(x-x_0)}{r}\right|\right)dz\\
&\qquad\le c_3\psi\left(T(r,R)^{1/2}\right)^{1-\theta_0}\left[\-int_{\Q_r(z_0)}\psi(|Du-A|)^{\theta_0}dz+\psi\left(\-int_{\Q_r(z_0)}\V'_{1+|A|}(|Du-A|)dz\right)^{\theta_0}\right],
\end{split}
\end{equation}
with $c_3<\infty$ depending on $n,N,p_0,p_1, q_0,q_1,\nu,L$.
Here
\begin{equation*}
T(r,R)=\-int_{\Q_R(z_0)}\left[\left|
    \frac{u-(u)_{z_0,R}-A(x-x_0)}{R-r}\right|^2
    +\V_{1+|A|}\left(\left|\frac{u-(u)_{z_0,R}-A(x-x_0)}{R-r}\right|\right)
    \right]dz.
\end{equation*}
\end{thm}

\begin{proof}
Without loss of generality we can assume that $z_0=(0,0)$. Suppose $\theta_0>0$ satisfies~\eqref{E:theta0}. We  use the Gagliardo-Nirenberg inequality (see \cite[Lemma 2.13]{ho} with $(\psi,\gamma,\theta,p,q_1,q_2)=(\psi^{\frac{1}{q_0}},q_0,\theta_0, \theta_0q_0,\frac{q_1}{q_0},2)$) to get
\begin{equation*}
\-int_{\B_r}\psi\left(\left|\frac{f}{r}\right|\right)dx\le c \left(\-int_{\B_r}\left[\psi(|Df|)^{\theta_0}+\psi\left(\left|\frac{f}{r}\right|\right)^{\theta_0}\right]dx\right)\psi\left(\left(\-int_{\B_r}\left|\frac{f}{r}\right|^2dx\right)^{\frac{1}{2}}\right)^{1-\theta_0}.
\end{equation*}
With $f=u-(u)_r-Ax$, we apply the previous inequality to each time slice:
\begin{equation*}
\begin{split}
&\-int_{\Q_r}\psi\left(\left|\frac{u-(u)_r-Ax}{r}\right|\right)dz\\
&\le c\left[\-int_{\Q_r}\left[\psi(|Du-A|)^{\theta_0}+\psi\left(\left|\frac{u-(u)_r-Ax}{r}\right|\right)^{\theta_0}\right]dz \right]\psi\left(\left(\sup_{-{r^2}<t<0}\-int_{\B_r}\left|\frac{u-(u)_r-Ax}{r}\right|^2dx\right)^{\frac{1}{2}}\right)^{1-\theta_0}.
\end{split}
\end{equation*}
Observe that $\psi^{\theta_0}$ satisfies Assumption~\ref{assumption} with exponents $1<\theta_0q_0\le\theta_0q_1$. It follows that $\Delta_2\left(\psi^{\theta_0},\left(\psi^{\theta_0}\right)^*\right)\le\max\left\{2^{\theta_0q_1},2^\frac{\theta_0q_0}{\theta_0q_0-1}\right\}<\infty$. We may therefore use the Poincar\'e type inequality in Theorem~\ref{poincare}, with $\psi$ replaced with $\psi^{\theta_0}$, to bound the second term in the right hand side. Thus,
\begin{equation*}
\begin{split}
&\-int_{\Q_r}\psi\left(\left|\frac{u-(u)_r-Ax}{r}\right|\right)dz\\
&\le c \left[\-int_{\Q_r}\psi(|Du-A|)^{\theta_0}dz+\psi\left(\-int_{\Q_r} \V'_{1+|A|}(|Du-A|)dz\right)^{\theta_0}\right]
\psi\left(\left(\sup_{-{r^2}<t<0}\-int_{\B_r}\left|\frac{u-(u)_r-Ax}{r}\right|^2dx\right)^{\frac{1}{2}}\right)^{1-\theta_0}.
\end{split}
\end{equation*}
Now, to estimate the sup-term, we apply the Caccioppoli inequality on the cylinders $\Q_r$ and $\Q_R$ (see Theorem \ref{cacc2}). Since $R\le 2r$,
\begin{equation*}
\begin{split}
\sup_{-{r^2}<t<0}\-int_{\B_r}\left|\frac{u-(u)_r-Ax}{r}\right|^2dx
&\le \frac{c_0}{r^{n+2}}\left[\int_{\Q_R}\V_{1+|A|}\left(\left|\frac{u-(u)_r-Ax}{R-r}\right|\right)+
\left|\frac{u-(u)_r-Ax}{R-r}\right|^2dz\right]\\
&\le c_0 2^{n+2}\left[\-int_{\Q_R}\V_{1+|A|}\left(\left|\frac{u-(u)_r-Ax}{R-r}\right|\right)+
\left|\frac{u-(u)_r-Ax}{R-r}\right|^2dz\right].
\end{split}
\end{equation*}
To replace $(u)_r$ with $(u)_R$, we note that
\begin{equation*}
|(u)_r-(u)_R|=\left|\-int_{\Q_r}[u-(u)_R-Ax]\,dz\right|\le 2^{n+2}\-int_{\Q_R}|u-(u)_R-Ax|\,dz.
\end{equation*}
The result follows from Jensen's inequality and the $\Delta_2$-condition.
\end{proof}

\noindent To establish the main regularity result, Theorem~\ref{main}, we will need the following inequality, which is proved using Theorem~\ref{sp} with the special choice $t\mapsto\psi(t)=t^2$.

\begin{thm}\label{sp2}
Let $u\in C^0(-T,0; L^2(\Omega,\R^N))\cap L^1(-T,0; W^{1,1}(\Omega,\R^N))$ be a weak solution to \eqref{P} satisfying $\V(|Du|)\in L^1(-T,0;L^1(\Omega))$. Under assumptions $(a_{1})$-$(a_{4})$ and Assumption~\ref{assumption}, given any standard cylinder $\Q_{\rho}(z_0)\subseteq\Omega_T$, with center $z_0=(x_0,t_0)$, and any $A\in\R^{Nn}$, we have
\begin{equation*}
\begin{split}
    &\-int_{\Q_{\rho/2}(z_0)}
        \left|\frac{u-(u)_{\rho/2}-A(x-x_0)}{\rho/2}
        \right|^2 dz\\
    &\qquad\le
    c_4\left[\-int_{\Q_{\rho}(z_0)}\V_{1+|A|}(|Du-A|)dz
        +\left(\-int_{\Q_{\rho}(z_0)}|Du-A|^{p_0}dz
        \right)^{\frac{2}{p_0}}\right.\\
    &\qquad\qquad\qquad\left.+
    \V_{1+|A|}\left(\-int_{\Q_{\rho}(z_0)}\V'_{1+|A|}(|Du-A|)dz
        \right)+\left(\-int_{\Q_{\rho}(z_0)}\V'_{1+|A|}(|Du-A|)dz
        \right)^{2}\right],
\end{split}
\end{equation*}
where $c_4$ depends on $n,N,p_0,p_1,\nu,$ and $L$.
\end{thm}
\begin{proof}
As usual, we may assume $z_0=(0,0)$. Let $\frac{\rho}{2}\le r<R\le \rho$. For convenience, put $M=1+|A|$. We use Theorem~\ref{sp} with the $N$-function $t\mapsto\psi(t)=t^2$, $q_0=q_1=2$, and $\theta_0=p_0/2$. Observe that, since $\frac{2n}{n+2}<p_0<2$,
\[
    \theta_0q_0=p_0\in (1,n)\qquad\text{ and }\qquad
    \frac{n}{n+2}
    =\frac{nq_1}{nq_1+2q_0}
    <\theta_0\le 1.
\]
The inequality~\eqref{E:sp} becomes
\begin{equation}\label{choicepsi}
\-int_{\Q_r}\left|\frac{u-(u)_r-Ax}{r}\right|^2 dz\\
\le c_3\left(T(r,R)\right)^{1-\frac{p_0}{2}}\left[\-int_{\Q_r}|Du-A|^{p_0}dz+\left(\-int_{\Q_r}\V'_{M}(|Du-A|)dz\right)^{p_0}\right],
\end{equation}
where
\begin{equation*}
\begin{split}
    T(r,R)
    &=
    \-int_{\Q_R}\left|\frac{u-(u)_R-Ax}{R-r}\right|^2
        +\V_{M}\left(\left|\frac{u-(u)_R-Ax}{R-r}\right|
        \right)dz\\
    &\le
    \left(\frac{R}{R-r}\right)^2\-int_{\Q_R}
        \left|\frac{u-(u)_R-Ax}{R}\right|^2dz
        +\left(\frac{R}{R-r}\right)^{p_1}
        \-int_{\Q_R}\V_{M}\left(\left|\frac{u-(u)_R-Ax}{R}
        \right|\right)dz\\
    &\le\left(\frac{R}{R-r}\right)^{p_1}\left[\-int_{\Q_R}
        \left|\frac{u-(u)_R-Ax}{R}\right|^2dz
        +\-int_{\Q_R}\V_{M}\left(\left|\frac{u-(u)_R-Ax}{R}
        \right|\right)dz\right].
\end{split}
\end{equation*}
Here, we have taken advantage of \eqref{t1}, $p_1>2$, and $\frac{R}{R-r}>1$.
Now, we use Young's inequality in \eqref{choicepsi}, with $\frac{1}{\theta_0}=\frac{2}{p_0}$ and its conjugate $\frac{1}{1-\theta_0}=\frac{2}{2-p_0}$. This produces
\begin{equation*}
\begin{split}
    \-int_{\Q_r}\left|\frac{u-(u)_r-Ax}{r}\right|^2 dz
    \le&
    \frac{1}{2}\-int_{\Q_R}
        \left|\frac{u-(u)_R-Ax}{R}\right|^2dz
        +\frac{1}{2}\-int_{\Q_R}\V_{M}
            \left(\left|\frac{u-(u)_R-Ax}{R}\right|\right)dz\\
    &\qquad+
    c \left(\frac{R}{R-r}\right)^{\frac{p_1(2-p_0)}{p_0}}
        \left[\-int_{\Q_r}|Du-A|^{p_0}dz
        +\left(\-int_{\Q_r}\V'_{M}(|Du-A|)dz\right)^{p_0}
        \right]^{\frac{2}{p_0}}.
\end{split}
\end{equation*}
For the second term in the upper bound, we can use Theorem \ref{poincare} with $\psi(t)=\V_{M}(t)$. The previous inequality  becomes
\begin{equation*}
\begin{split}
    &\-int_{\Q_r}\left|\frac{u-(u)_r-Ax}{r}\right|^2 dz\\
    &\le
    \frac{1}{2}\-int_{\Q_R}
        \left|\frac{u-(u)_R-Ax}{R}\right|^2dz
        +c\-int_{\Q_R}\V_{M}(|Du-A|)dz
        +\V_{M}\left(\-int_{\Q_R}\V'_{M}(|Du-A|)dz\right)\\
    &\quad\quad+
    c \left(\frac{R}{R-r}\right)^{\frac{p_1(2-p_0)}{p_0}}
        \left[\-int_{\Q_r}|Du-A|^{p_0}dz
    +\left(\-int_{\Q_r}\V'_{M}(|Du-A|)dz\right)^{p_0}
        \right]^{\frac{2}{p_0}}.
\end{split}
\end{equation*}
Enlarging the domain of integration (recall that $\frac{\rho}{2}\le r<R\le\rho$), we get
\begin{equation*}
\begin{split}
    &\-int_{\Q_r}\left|\frac{u-(u)_r-Ax}{r}\right|^2 dz\\
    &\le
    \frac{1}{2}\-int_{\Q_R}
        \left|\frac{u-(u)_R-Ax}{R}\right|^2dz
        +c\-int_{\Q_{\rho}}\V_{M}(|Du-A|)dz
        +\V_{M}\left(\-int_{\Q_{\rho}}\V'_{M}(|Du-A|)dz\right)\\
    &\quad\quad+
    c \left(\frac{\rho}{R-r}\right)^{\frac{p_1(2-p_0)}{p_0}}
        \left[\-int_{\Q_{\rho}}|Du-A|^{p_0}dz
        +\left(\-int_{\Q_{\rho}}\V'_{M}(|Du-A|)dz\right)^{p_0}
        \right]^{\frac{2}{p_0}}.
\end{split}
\end{equation*}
We are now in position to apply \cite[Lemma 6.1]{Giu} to conclude the proof.
\end{proof}


\section{Linearization}\label{lin}

\noindent
We now prove a lemma that facilitates the comparison of solutions to our system~\eqref{P} to solutions for a linear system with constant coefficients. 
\begin{lem}\label{linear}
Suppose $u\in C^0(-T,0; L^2(\Omega,\R^N))\cap L^1(-T,0; W^{1,1}(\Omega,\R^N))$ is a weak solution to  \eqref{P} that satisfies $\V(|Du|)\in L^1(-T,0; L^1(\Omega))$. Let a generic parabolic cylinder $\Q_{\rho,\tau}(z_0)\subset\Omega_T$, with center $z_0=(x_0,t_0)$, be given. Under the hypotheses $(a_{1})$-$(a_{4})$ and Assumption \ref{assumption}, given any affine map $\ell: \R^{n}\to \R^{N}$ and any map $\eta\in C^\infty_c(\Q_{\rho/2,\tau/4}(z_0),\R^N)$, we have 
\begin{equation*}
\left| \-int_{\Q_{\rho/2,\tau/4}(z_{0})} (u-\ell)\cdot \eta_t- Da(D\ell)( Du-D\ell, D\eta) \, dz \right| \leq c_5\V_{1+|D\ell|}(1) \left\{\omega \left(S^{\frac{1}{2}}\right)^{\frac{1}{2}}S^{\frac{1}{2}}+S\right\} \sup_{\Q_{\rho/2,\tau/4}(z_{0})}|D\eta|
\end{equation*}
where $c_5$ depends only on $n, N, L,\nu, p_0,p_1$. Here
\[
    S=\-int_{\Q_{\rho/2,\tau/4}(z_0)}
    \frac{\V_{1+|D\ell|}(|Du-D\ell|)}{\V_{1+|D\ell|}(1)}\,dz.
\]
\end{lem}

\begin{proof}

We may assume $z_0=(0,0)$. For convenience, we write $\Q_{\rho,\tau}=\Q_{\rho,\tau}(z_0)$ and $M=1+|D\ell|$.  Let $\eta\in C^\infty_c(\Q_{\rho/2,\tau/4},\R^N)$ be given. We first note:
\begin{equation*}
\begin{split}
    &\-int_{\Q_{\rho/2,\tau/4}}
        [(u-\ell)\cdot \eta_t-Da(D\ell)(Du-D\ell,D\eta)]\,dz\\
    &=
    \-int_{\Q_{\rho/2,\tau/4}}[(u-\ell)\cdot \eta_t-a(Du)\cdot D\eta]\,dz\\
    &\quad +
    \-int_{\Q_{\rho/2,\tau/4}}[(a(Du)-a(D\ell))\cdot 
        D\eta-Da(D\ell)(Du-D\ell,D\eta)]\,dz
    =:I+I\!I.
\end{split}
\end{equation*}
Since $u$ is a weak solution to \eqref{P} and $\displaystyle{\-int_{\Q_{\rho/2,\tau/4}}\ell\cdot \eta_t\,dz=0}$, we infer that $I=0$. On the other hand,
\begin{equation*}
\begin{split}
\-int_{\Q_{\rho/2,\tau/4}}(a(Du)-a(D\ell))\cdot D\eta\,dz&=\-int_{\Q_{\rho/2,\tau/4}}\int_0^1\frac{d}{ds}
    \left[a(D\ell+s(Du-D\ell))\right]\cdot D\eta\,ds\,dz\\
&=\-int_{\Q_{\rho/2,\tau/4}}\int_0^1Da(D\ell+s(Du-D\ell))(Du-D\ell,D\eta)\,ds\,dz,
\end{split}
\end{equation*}
so that 
$$
I\!I=\-int_{\Q_{\rho/2,\tau/4}}\int_0^1[Da(D\ell+s(Du-D\ell))-Da(D\ell)](Du-D\ell,D\eta)\,ds\,dz.
$$
Using the continuity assumption $(a_4)$,
\begin{equation*}
\begin{split}
|I\!I|
    &\le
    L\-int_{\Q_{\rho/2,\tau/4}}\int_0^1
        \omega\left(\frac{s|Du-D\ell|}{1+|D\ell+s(Du-D\ell)|+|D\ell|}\right)
        \V''(M+|D\ell+s(Du-D\ell)|)|Du-D\ell||D\eta|\,ds\,dz\\
    &\le
    L\-int_{\Q_{\rho/2,\tau/4}}
        \omega\left(|Du-D\ell|\right)
        |Du-D\ell||D\eta|\int_0^1\V''(M+|D\ell+s(Du-D\ell)|)\,ds\,dz.
\end{split}
\end{equation*}
As $\V'$ is nondecreasing, Assumption \ref{assumption}, Lemma \ref{lem2}, and~\eqref{t2} yield
$$
\int_0^1\V''(M+|D\ell+s(Du-D\ell)|)\,ds
\le c\,\frac{\V'(M+|Du|+|D\ell|)}{M+|Du|+|D\ell|}
\le c\,\frac{\V'(M+|Du|)}{M+|Du|}.
$$
Thus,
$$
|I\!I|\le c\sup_{\Q_{\rho/2,\tau/4}}\!\!\!\!|D\eta|\,\-int_{\Q_{\rho/2,\tau/4}}
\omega\left(|Du-D\ell|\right)|Du-D\ell|\,
    \frac{\V'(M+|Du|)}{M+|Du|}\,dz.
$$
Now, we distinguish in $\Q_{\rho/2,\tau/4}$ the points where $|Du-D\ell|\le M$ from those where $|Du-D\ell|>M$. Denote by $\mathbb{X}$ the first set and by $\mathbb{Y}$ the second. On $\mathbb{X}$, we have $|Du|\le2M$, so~\eqref{t2} implies $\V'(M+|Du|)\sim\V'(M)$. Thus
\begin{align*}
|Du-D\ell|\frac{\V'(M+|Du|)}{M+|Du|} & \leq c\left(\frac{\V'(M)}{M}\right)^\frac{1}{2} \left(|Du-D\ell|^2\frac{\V'(M+|Du-D\ell|)}{M+|Du-D\ell|}\right)^\frac{1}{2} \\
&\leq c\left(\frac{\V'(M)}{M}\right)^\frac{1}{2}\V_M(|Du-D\ell|)^\frac{1}{2},
\end{align*}
where we have used~\eqref{t6}. Moreover,
\begin{align*}
|Du-D\ell| &\leq c\frac{M+|Du|}{\V'(M+|Du|)}\left(\frac{\V'(M)}{M}\right)^\frac{1}{2}\V_M(|Du-D\ell|)^\frac{1}{2}\\
&\leq c\left(\frac{M}{\V'(M)}\right)^\frac{1}{2}\V_M(|Du-D\ell|)^\frac{1}{2}.
\end{align*}
It follows that
\begin{align*}
    &\-int_{\Q_{\rho/2,\tau/4}}\!\!\!\!\!\!\chi_{\mathbb X}
    \omega\left(|Du-D\ell|\right)|Du-D\ell|\,
    \frac{\V'(M+|Du|)}{M+|Du|}\,dz\\
    &\qquad\le
    c\left(\frac{\V'(M)}{M}\right)^\frac{1}{2}\-int_{\Q_{\rho/2,\tau/4}}
    \omega\left(\left(\frac{M}{\V'(M)}\right)^\frac{1}{2}
        \V_M(|Du-D\ell|)^\frac{1}{2}\right)
    \V_M(|Du-D\ell|)^\frac{1}{2}\,dz. 
\end{align*}
Applying H\"{o}lder's inequality and using the concavity of $\omega$ and the bound $\omega\le1$, we continue with
\begin{align*}
&\-int_{\Q_{\rho/2,\tau/4}}\chi_{\mathbb X} \omega\left(|Du-D\ell|\right)|Du-D\ell|\,\frac{\V'(M+|Du|)}{M+|Du|}\,dz\\
&\leq c\left(\frac{\V'(M)}{M}\right)^\frac{1}{2} \left(\-int_{\Q_{\rho/2,\tau/4}} \omega\left(\left(\frac{M}{\V'(M)}\right)^\frac{1}{2}  \V_M(|Du-D\ell|)^\frac{1}{2}\right)^2\,dz\right)^\frac{1}{2} \left(\-int_{\Q_{\rho/2,\tau/4}}\V_M(|Du-D\ell|)\,dz\right)^\frac{1}{2}\\
&\leq c\left(\frac{\V'(M)}{M}\right)^\frac{1}{2} \omega\left(\left(\frac{M}{\V'(M)}\right)^\frac{1}{2} \left( \-int_{\Q_{\rho/2,\tau/4}}\V_M(|Du-D\ell|)\,dz\right)^\frac{1}{2}\right)^{\frac{1}{2}}\left(\-int_{\Q_{\rho/2,\tau/4}}\V_M(|Du-D\ell|)\,dz\right)^\frac{1}{2}\\
&\leq c\left(\frac{\V'(M)}{M}\right)^\frac{1}{2} \omega\left(\left(\frac{M}{\V'(M)}\right)^\frac{1}{2}S^\frac{1}{2}\V_M(1)^{\frac{1}{2}}\right)^{\frac{1}{2}} S^\frac{1}{2}\V_M(1)^{\frac{1}{2}}.
\end{align*}
To complete the bound on $\mathbb{X}$, we use~\eqref{t2} and~\eqref{t6} to deduce that
\[
    \frac{\V'(M)}{M}\sim\frac{\V'(M+1)}{M+1}\sim\V_M(1).
\]
We conclude that
\begin{equation}\label{E:Linear1}
    \-int_{\Q_{\rho/2,\tau/4}}\!\!\!\!\!\!\chi_{\mathbb X}
    \omega\left(|Du-D\ell|\right)|Du-D\ell|\,
    \frac{\V'(M+|Du|)}{M+|Du|}\,dz
    \le
    c\V_M(1)\omega\left(S^\frac{1}{2}\right)^\frac{1}{2}
        S^\frac{1}{2}.
\end{equation}
Turning to the set $\mathbb{Y}$, we have $1\le M\le|Du-D\ell|\le|Du-D\ell|^2$. Recalling~\eqref{t2} and~\eqref{t6} once more, we have
\[
    |Du-D\ell|\frac{\V'(M+|Du|)}{M+|Du|}
    \le
    c|Du-D\ell|^2\frac{\V'(M+|Du-D\ell|)}{M+|Du-D\ell|}
    \le
    c\V_M(|Du-D\ell|).
\]
Thus
\begin{equation}\label{E:Linear2}
    \-int_{\Q_{\rho/2,\tau/4}}\!\!\!\!\!\!\chi_{\mathbb Y}
    \omega\left(|Du-D\ell|\right)|Du-D\ell|\,
    \frac{\V'(M+|Du|)}{M+|Du|}\,dz
    \le
    c\-int_{\Q_{\rho/2,\tau/4}}\V_M(|Du-D\ell|)\,dz
    =c\V_M(1)S.
\end{equation}
Here, we have again used $\omega\le1$. The lemma follows from summing~\eqref{E:Linear1} and~\eqref{E:Linear2}. To verify the claim for the parameter dependencies of the constant $c_5$, we review the proof and note that only the hypothesis $(a_4)$, Assumption~\ref{assumption}, and properties~\eqref{t2} and~\eqref{t6} were required. 
\end{proof}

\section{Decay Estimates}


\noindent
For convenience, we recall the excess functional introduced in Section~\ref{S:Intro}. Given $z_0\in\Omega_T$, $a\ge0$, $r>0$, and an affine map $\ell:\R^{n}\to\R^{N}$, define
\begin{equation*}
\Psi_a(z_0,r,\ell)=\-int_{\Q_r(z_0)}\left(\left|\frac{u-\ell}{r}\right|^2+\V_a\left(\left|\frac{u-\ell}{r}\right|\right)\right)\, d z.
\end{equation*}

\noindent In the following lemma we provide the decay of the excess $\Psi_{1+|D\ell_{z_0,\rho}|}(z_0,\rho,\ell_{z_0,\rho})$. Recall that $\ell_{z_0,\rho}$ is defined in~\eqref{E:AffineMin} and denotes the time-independent affine map closest to $u$ with respect to the $L^2$-norm on $\Q_\rho(z_0)$.

\begin{lem}[Decay Estimate]\label{L:DecEst}
Suppose that hypotheses $(a_{1})$-$(a_{4})$ and Assumption \ref{assumption} hold and that $u\in C^0(-T,0; L^2(\Omega,\R^N))\cap L^1(-T,0; W^{1,1}(\Omega,\R^N))$ is a weak solution to  \eqref{P} satisfying $\V(|Du|)\in L^1(-T,0; L^1(\Omega))$. Let $M_0>0$ and $0<\alpha<1$ be given. There exist $0<\e,\theta<1$ with the following property: if $z_0\in\Omega_T$ and $\rho>0$ are such that $\Q_\rho(z_0)\subseteq\Omega_T$,
\begin{equation*}
|D\ell_{z_0,\rho}|\le M_0,
\qquad\text{ and }\qquad\Psi_{1+|D\ell_{z_0,\rho}|}(z_0,\rho,\ell_{z_0,\rho})\le\e,
\end{equation*}
then
\begin{equation*}
\Psi_{1+|D\ell_{z_0,\theta\rho}|}(z_0,\theta\rho,\ell_{z_0,\theta\rho}) \le \theta^{2\alpha}\Psi_{1+|D\ell_{z_0,\rho}|}(z_0,\rho,\ell_{z_0,\rho})
\end{equation*}
and
\begin{equation*}
|D\ell_{z_0,\theta\rho}| \le |D\ell_{z_0,\rho}| +(n+2)\left(\frac{1}{\theta}\right)^{n+3}\Psi_{1+|D\ell_\rho|}(z_0,\rho,\ell_{z_0,\rho})^\frac{1}{2}.
\end{equation*}
\end{lem}

\begin{proof}
For each $0<r\le\rho$, we write $\ell_r=\ell_{z_0,r}$ and $\Psi_a(r)=\Psi_a(z_0,r,\ell_r)$. Define $v=u-\ell_\rho$ and $M=1+|D\ell_\rho|\le M_0+1$. Let $\A$ denote the bilinear form $Da(D\ell_{\rho})$. Thus
\begin{equation*}
    \A(\xi,\xi)\ge\nu\V''(1+|D\ell_\rho|)|\xi|^2
    \quad \mbox{ and } \quad
    |\A(\xi,\eta)|\le L\V''(1+|D\ell_\rho|)|\xi||\eta|.
\end{equation*}

\noindent
Let us recall that  $\Delta_2(\V_M)\le 2^{p_1}$; with $0<\alpha<1$ given, define 
\begin{equation}\label{E:Theta}
    0<\theta=\theta(M,\alpha)
    =\min\left\{
        \frac{1}{32},
        \left(\frac{1}{1+c_{6,6}}\right)^\frac{1}{2(1-\alpha)}
    \right\},
\end{equation}
and suppose that
\begin{equation}\label{E:Small}
    \Psi_M(\rho)\le\e=\e(M)
    \le\min\left\{\frac{\V_M(1)}{c_{6,1}},\frac{\delta^2}{c_{6,2}^2},
        \frac{1}{\omega_nc_{6,3}c_{6,4}},\left(\frac{\theta^{n+3}}{n+2}\right)^2,
        \frac{(4\theta)^{n+p_1+4}}{c_{6,3}c_{6,5}\left(2+\Delta_2(\V_M)^3\right)}\right\},
\end{equation}
where the precise values of the constants $c_{6,i}>1$ will be determined in the course of the proof.
The constant $0<\delta<1$ is specified by Theorem~\ref{T:ApproxLemma}, while the constant $C$ appearing below may change from line to line but will depend on only $n,N,L/\nu,p_0,p_1$.

\noindent
As $\ell_\rho$ is independent of $t$, the map $v$ is a weak solution to~\eqref{P}. Our first objective is to take advantage of the $\A$-caloric approximation lemma to produce an $\A$-caloric map close to $v$. With $S$ defined in Lemma~\ref{linear}, the Caccioppoli inequality in Corollary~\ref{cacc3} implies
\[
    S=\-int_{Q_{\rho/2}}\frac{\V_M(|Du-D\ell_\rho|)}{\V_M(1)}\,dz
    \le
    \frac{c_02^{n+p_1+2}}{\V_M(1)}
    \-int_{Q_\rho}\left[\left|\frac{v}{\rho}\right|^2
        +\V_M\left(\left|\frac{v}{\rho}\right|\right)\right]\,dz
    \le
    \frac{c_{6,1}}{\V_M(1)}\Psi_M(\rho)\le 1.
\]
Note that $c_{6,1}=c_02^{n+p_1+2}$. Now, Lemma~\ref{linear} delivers the bound
\begin{align*}
    \left|\-int_{\Q_{\rho/2}}
        \left( v\cdot\partial_t\eta-\A(Dv,D\eta)\right)\, dz\right|
    &\le
    c_5\V_M(1)\left\{\omega\left(S^\frac{1}{2}\right)^{\frac{1}{2}}S^\frac{1}{2}+S\right\}
        \sup_{Q_{\rho/2}}|D\eta|\\
    &\le
    2c_5c_{6,1}^\frac{1}{2}\V_M(1)^\frac{1}{2}\Psi_M(\rho)^\frac{1}{2}\sup_{Q_{\rho/2}}|D\eta|\\
    &=c_{6,2}\Psi_M(\rho)^\frac{1}{2}\sup_{Q_{\rho/2}}|D\eta|\\
    &\le 
    \delta\sup_{Q_{\rho/2}}|D\eta|.
\end{align*}
The smallness condition~\eqref{E:Small} was applied in the last inequality. This verifies the requirement in~\eqref{E:ACaloricBnd} of Theorem~\ref{T:ApproxLemma}. For the other requirement in~\eqref{E:NormBnd}, we again use the Caccioppoli inequality:
\begin{align*}
     &\sup_{t_0-\rho^2/4<t<t_0}\-int_{\B_{\rho/2}}
        \left|\frac{v}{\rho/2}\right|^2\,dx
    +\-int_{\Q_{\rho/2}}
        \left(\V_M\left(\left|\frac{v}{\rho/2}\right|\right)
        +\V_{M}(|Dv|)\right)\,dz\\
    &\qquad\qquad\qquad\qquad\le
    4c_{6,1}\Psi_M(\rho)+\Delta_2(\V_M)\-int_{\Q_{\rho/2}}
        \V_M\left(\left|\frac{v}{\rho}\right|\right)\,dz\\
    &\qquad\qquad\qquad\qquad\le c_02^{n+p_1+5}\Psi_M(\rho)=
    c_{6,3}\Psi_M(\rho)=:\gamma^2\le1/\omega_n\le 1.
\end{align*}

\noindent
With the hypotheses of Theorem~\ref{T:ApproxLemma} satisfied, taking into account Remarks~\ref{phiap0p1} and~\ref{R:ApproxLemma}, we obtain an $\A$-caloric map $h\in C^\infty(\Q_{\rho/4};\R^{N})$ such that
\begin{equation}\label{E:OrliczApp1}
    \-int_{\Q_{\rho/4}}\left(\left|\frac{\gamma h}{\rho/4}\right|^2
    +\V_{M}\left(\left|\frac{\gamma h}{\rho/4}\right|\right)
    +\V_{M}(|\gamma Dh|)\right)\, d z
    \le
    2^{n+2}K_{\V_M}\gamma^2
\end{equation}
and
\begin{equation*}
    \-int_{\Q_{\rho/4}}\left(\left|\frac{v-\gamma h}{\rho/4}\right|^2
    +\V_{M}\left(\left|\frac{v-\gamma h}{\rho/4}\right|\right)\right)\, d z
    \le\e\gamma^2.
\end{equation*}

\noindent
Recall that $0<\theta<1/16$ by definition~\eqref{E:Theta}. As in Lemma~\ref{L:CaloricDecay}, for $0<r\le\rho/4$, we define the affine map $\ell^{(h)}_{r}(x)=(h)_{z_0,r}+(Dh)_{z_0,r}(x-x_0)$. We want to produce a bound for
\begin{equation*}
    \-int_{\Q_{\theta\rho}}\left|
        \frac{v-\gamma\ell^{(h)}_{\theta\rho}}{\theta\rho}\right|^2\, d z
    +\-int_{\Q_{\theta\rho}}\V_{M}\left(
        \left|\frac{v-\gamma\ell^{(h)}_{\theta\rho}}{\theta\rho}\right|\right)\, d z
    =:I+I\!I.
\end{equation*}
We will focus on $I\!I$. The argument for $I$ is similar. We may write
\begin{align}
\nonumber
    I\!I
    &\le
    \Delta_2(\V_M)\-int_{\Q_{\theta\rho}}\V_M\left(\left|\frac{v-\gamma h}{\theta\rho}\right|\right)
        +\V_{M}\left(\left|\frac{\gamma(h-\ell^{(h)}_{z_0,\theta\rho})}{\theta\rho}\right|\right)\, d z\\
\nonumber
    &\le
    \Delta_2(\V_M)\left(\frac{1}{4\theta}\right)^{p_1+n+2}
    \-int_{\Q_{\rho/4}}\V_{M}\left(
        \left|\frac{v-\gamma h}{\rho/4}\right|\right)\, d z
        +\Delta_2(\V_M)\-int_{\Q_{\theta\rho}}\V_{M}\left(
        \left|\frac{\gamma (h-\ell^{(h)}_{z_0,\theta\rho})}{\theta\rho}\right|\right)\, d z\\
\label{E:IIExcessBnd1}
    &\le
    \Delta_2(\V_M)\left(\frac{1}{4\theta}\right)^{p_1+n+2}\e\gamma^2
    +\Delta_2(\V_M)\-int_{\Q_{\theta\rho}}\V_{M}\left(
        \left|\frac{\gamma (h-\ell^{(h)}_{z_0,\theta\rho})}{\theta\rho}\right|\right)\, d z.
\end{align}
Observe that Lemma~\ref{L:OrliczPower}-$(b)$ and~\eqref{E:OrliczApp1} imply
\begin{align*}
    \-int_{\Q_{\rho/4}}
        \left|\frac{\gamma(h-\ell^{(h)}_{z_0,\rho/4})}{\rho/4}\right|^2\, d z
    \le&
    6\-int_{\Q_{\rho/4}}\left|\frac{\gamma h}{\rho/4}\right|^2\, d z
    +3\left(\-int_{\Q_{\rho/4}}|\gamma Dh|\right)^2\, d z\\
    \le &
    2^{n+5}K_{\V_M}\gamma^2
    +\frac{C}{\V''(M)}\-int_{\Q_{\rho/4}}
        \V_{M}(|\gamma Dh|)\, d z+\frac{C}{\V_{M}(1)^2}\left(\-int_{\Q_{\rho/4}}
        \V_{M}(|\gamma Dh|)\, d z\right)^2\\
    \le &
    2^{n+5}K_{\V_M}\gamma^2+C\left(\frac{1}{\V''(M)}
        +\frac{2^{n+2}K_{\V_M}\gamma^2}{\V_M(1)^2}\right)2^{n+2}K_{\V_M}\gamma^2\\
    \le&2^{2n+5}K^2_{\V_M}\left[1+C\left(\frac{1}{\V''(M)}
        +\frac{1}{\V_M(1)^2}\right)\right]\gamma^2=c_{6,4}\gamma^2\le 1.
\end{align*}
We may therefore use the inequality in Lemma~\ref{L:CaloricDecay}, Remark \ref{phiap0p1}, (\ref{phiaquadratic2}), and Jensen's inequality to deduce that
\begin{equation}\label{E:hPhiExcess2}
\-int_{\Q_{\theta\rho}}
    \V_{M}\left(\left|\frac{\gamma(h-\ell^{(h)}_{z_0,\theta\rho})}{\theta\rho}\right|
        \right)\, d z
    \le C\theta^{2}\-int_{Q_{\rho/4}}
    \V_{M}
        \left(\left|\frac{\gamma(h-\ell^{(h)}_{z_0,\rho/4})}{\rho/4}\right|\right)\, d z.
\end{equation}
On the other hand Jensen's inequality and the $\Delta_2$-property imply
\begin{align}
\nonumber
    \-int_{\Q_{\rho/4}}\V_{M}\left(\left|
        \frac{\gamma(h-\ell^{(h)}_{z_0,\rho/4})}{\rho/4}\right|\right)\, d z
    \le&
    \Delta_2(\V_M)^2\-int_{\Q_{\rho/4}}\left(
        \V_{M}\left(\left|\frac{\gamma h}{\rho/4}\right|\right)
        +\V_{M}(|\gamma Dh|)\right)\, d z\\
\label{E:hPhiExcess1}
    \le&
    2^{n+2}K_{\V_M}\Delta_2(\V_M)^2\gamma^2.
\end{align}
With (\ref{E:hPhiExcess2}) and (\ref{E:hPhiExcess1}) we return to~\eqref{E:IIExcessBnd1} to obtain
\begin{equation*}
    \-int_{\Q_{\theta\rho}}\V_M\left(\left|
        \frac{v-\gamma\ell^{(h)}_{\theta\rho}}{\theta\rho}\right|\right)\,dz
    =II
    \le
    \Delta_2(\V_M)^3\left[\left(\frac{1}{4\theta}\right)^{p_1+n+2}\e
        +C2^{n+2}K_{\V_M}\theta^2\right]\gamma^2.
\end{equation*}
We similarly obtain
\begin{equation*}
    \-int_{\Q_{\theta\rho}}\left|
        \frac{v-\gamma\ell^{(h)}_{\theta\rho}}{\theta\rho}\right|^2\,dz
    =I
    \le
    \left[2\left(\frac{1}{4\theta}\right)^{n+4}\e
        +Cc_{6,4}\theta^2\right]\gamma^2.
\end{equation*}
Thus, since $\theta\le1/4$,
\begin{align*}
    & \-int_{\Q_{\theta\rho}}\left|
        \frac{v-\gamma\ell^{(h)}_{z_0,\theta\rho}}{\theta\rho}\right|^2
    +\V_{M}\left(
        \left|\frac{v-\gamma\ell^{(h)}_{z_0,\theta\rho}}{\theta\rho}
        \right|\right)\, d z \le
    \left[\left(\frac{1}{4\theta}\right)^{n+p_1+2}
        \left(2+\Delta_2(\V_M)^3\right)\e
    +Cc_{6,4}\left(1+\Delta_2(\V_M)^3\right)\theta^2\right]\gamma^2.
\end{align*}
Let us point out that $1+|D\ell_{\theta\rho}|\le M+1$. Indeed, from \eqref{touse} we obtain
\begin{equation}\label{ellthetarho}
\begin{split}
    |D\ell_{\theta\rho}|
    &\le
    |D\ell_\rho|+|D\ell_{\rho}- D\ell_{\theta\rho}|\\
    &\le
    M-1+(n+2) \-int_{\Q_{\theta\rho}}
        \left|\frac{u- \ell_{\rho}}{\theta\rho}\right|\, dz \leq M-1+\left(\frac{n+2}{\theta^{n+3}}\right)
        \-int_{\Q_{\rho}} \left|\frac{u- \ell_{\rho}}{\rho}\right|\, dz\\
    &\le
    M-1+\left(\frac{n+2}{\theta^{n+3}}\right)\left(\-int_{\Q_{\rho}}
        \left|\frac{u- \ell_{\rho}}{\rho}\right|^2\, dz\right)^{\frac{1}{2}}\\
    &\le
    M-1+\left(\frac{n+2}{\theta^{n+3}}\right)
        \Psi_{1+|D\ell_{\rho}|}(\rho)^{\frac{1}{2}}\\
    &\le
    M-1+\left(\frac{n+2}{\theta^{n+3}}\right)\varepsilon^{\frac{1}{2}}
    \le M
\end{split}
\end{equation}
provided $\displaystyle{\varepsilon\le \left(\frac{\theta^{n+3}}{n+2}\right)^2}$. 

\noindent
Now, using Lemma \ref{L:ShiftComp} (with $M$ replaced by $M+1$) and Lemma \ref{quasimin} (with $\ell=\ell_\rho-\gamma\ell^{(h)}_{z_0,\theta\rho}$), and defining $c_{6,5}=\kappa_0\,4^{p_1+1}(M+1)^{p_1+2}$ ($\kappa_0$ from Lemma  \ref{quasimin}), we have
\begin{align*}
    \Psi_{1+|D\ell_\theta\rho|}(z_0,\theta\rho, \ell_{\theta\rho})
    =&
    \-int_{\Q_{\theta\rho}}\left(\left|
            \frac{u-\ell_{\theta\rho}}{\theta\rho}\right|^2
        +\V_{1+|D\ell_{\theta\rho}|}\left(\left|
            \frac{u-\ell_{\theta\rho}}{\theta\rho}\right|\right)\right)\, d z\\
    \le &
    4^{p_1+1}(M+1)^{p_1+2}
    \-int_{\Q_{\theta\rho}}\left(\left|
            \frac{u-\ell_{\theta\rho}}{\theta\rho}\right|^2
        +\V_{M}\left(\left|
            \frac{u-\ell_{\theta\rho}}{\theta\rho}\right|\right)\right)\, d z\\
    \le&
    c_{6,5}\-int_{\Q_{\theta\rho}}\left(\left|
    \frac{u-\ell_\rho-\gamma\ell^{(h)}_{z_0,\theta\rho}}{\theta\rho}\right|^2
        +\V_{M}\left(\left|
            \frac{u-\ell_\rho-\gamma\ell^{(h)}_{z_0,\theta\rho}}{\theta\rho}
        \right|\right)\right)\, d z\\
    =&
    c_{6,5}\-int_{\Q_{\theta\rho}}\left(\left|
    \frac{v-\gamma\ell^{(h)}_{z_0,\theta\rho}}{\theta\rho}\right|^2
        +\V_{M}\left(\left|
            \frac{v-\gamma\ell^{(h)}_{z_0,\theta\rho}}{\theta\rho}
        \right|\right)\right)\, d z\\
    \le&
    c_{6,3} \, c_{6,5} \left[\left(\frac{1}{4\theta}\right)^{n+p_1+2}
        \left(2+\Delta_2(\V_M)^3\right)\e
    +Cc_{6,4}\left(1+\Delta_2(\V_M)^3\right)\theta^2\right]\Psi_M(\rho)\\
    \le&
    \left[\theta^2+c_{6,6}\theta^2\right]\Psi_{M}(\rho),
    \end{align*}
    where $c_{6,6}=Cc_{6,3}c_{6,4}c_{6,5}(1+\Delta_2(\V_M)^3)$.
So, under the smallness assumption that $\Psi_{M}(\rho)\le\e$,
\[
    \Psi_{1+|D\ell_{\theta\rho}|}(z_0,\theta\rho,\ell_{\theta\rho})
    \le    \theta^{2\alpha}\Psi_{1+|D\ell_\rho|}(z_0,\rho,\ell_\rho).
\]
Combined with \eqref{ellthetarho}, we conclude that
\[
    |D\ell_{\theta\rho}|
    \le
    |D\ell_\rho|
    +(n+2)\left(\frac{1}{\theta}\right)^{n+3}\Psi_{1+|D\ell_\rho|}(\rho)^\frac{1}{2}.
\]
\end{proof}

\noindent
In the following lemma we will iterate the excess-decay estimate from the previous lemma.

\begin{lem}[Iteration Argument]\label{iteration}
Suppose that the assumptions, for $u$ and $\V$, in Lemma~\ref{L:DecEst} hold. Let $M_0>1$ and $0<\alpha<1$ be given. There exist $0<\e_0<\theta_0<1$ and $c_7=c_7(M_0,\theta_0,n,N,L/\nu,p_0,p_1)$ with the following property: given a standard parabolic cylinder $\Q_{\rho_0}(z_0)\subseteq\Omega_T$, if
\begin{equation*}
    1+|D\ell_{z_0,\rho_0}|\le M_0
\qquad\text{ and }\qquad
    \Psi_{1+|D\ell_{z_0,\rho_0}|}(z_0,\rho,\ell_{z_0,\rho_0})\leq\e_0,
\end{equation*}
then for each $j\in\N$, we have the following:
\begin{compactenum}[(a)]
    \item $\Psi_{1+|D\ell_{z_0,\theta_0^{j}\rho_0}|}(z_0,\theta_0^j\rho_0,\ell_{z_0,\theta_0^j\rho_0})
    \le
    \theta_0^{2j\alpha}\Psi_{1+|D\ell_{z_0,\rho_0}|}(z_0,\rho_0,\ell_{z_0,\rho_0})$,
    \item $|D\ell_{z_0,\theta_0^j\rho_0}|\le M_0-\theta_0^{j\alpha}$.
\end{compactenum}
Moreover,
\begin{equation}\label{E:CampEst}
    \-int_{Q_r(z_0)}
    \V_{1+|(Du)_{z_0,r}|}(|Du-(Du)_{z_0,r}|)\, d z
    \le
    c_7\left(\frac{r}{\rho_0}\right)^{2\alpha}
    \Psi_{1+|D\ell_{z_0,\rho_0}|}(z_0,\rho_0,\ell_{z_0,\rho_0}),
    \quad\text{ for all } 0<r\le\rho_0/2.
\end{equation}
\end{lem} 

\begin{proof}
Parts $(a)$ and $(b)$ follow from an induction argument. With $0<\alpha<1$, $\rho=\rho_0$, and $M_0$ fixed, let $0<\e,\theta<1$ be provided by Lemma~\ref{L:DecEst}. Put
\begin{equation}\label{E:IterParam}
    \theta_0:=\theta,\qquad
    \e_0:=\min\left\{\e,\frac{\theta_0^{2(n+3)}(1-\theta_0^\alpha)^2}{(n+2)^2}\right\},\qquad\text{ and }\qquad
    \rho_j=\theta_0^{j}\rho_0,\quad\text{ for each }j\in\mathbb{N}.
\end{equation}
Clearly $|D\ell_{z_0,\rho_0}|\le M_0$, and the base case, $j=1$, immediately follows from Lemma~\ref{L:DecEst}. With $j\in\N$ given, suppose that $(a)$ and $(b)$ are both true for all $k=1,\dots,j$. We observe that $(a)$ implies
\begin{equation*}
    \Psi_{1+|D\ell_{z_0,\rho_j}|}(z_0,\rho_j,\ell_{z_0,\rho_j})
    \le
    \theta_0^{j\alpha}\Psi_{1+|D\ell_{z_0,\rho_0}|}(z_0,\rho_0,\ell_{z_0,\rho_0})
    \le\e_0\le\e.
\end{equation*}
By the inductive assumption,
\[
   |D\ell_{z_0,\rho_j}|\le M_0-\theta_0^{j\alpha}\le M_0,
\]
We may therefore use Lemma~\ref{L:DecEst}, with $\rho$ replaced with $\rho_j$ and the other parameters the same as in~\eqref{E:IterParam}, to obtain
\begin{align*}
    \Psi_{1+|D\ell_{z_0,\rho_{j+1}}|}(z_0,\rho_{j+1},\ell_{z_0,\rho_{j+1}})
    \le&
    \theta_0^{2\alpha}\Psi_{1+|D\ell_{z_0,\rho_j}|}(z_0,\rho_j,\ell_{z_0,\rho_j})\\
    \le&
    \theta_0^{2(j+1)\alpha}\Psi_{1+|D\ell_{z_0,\rho_0}|}(z_0,\rho_0,\ell_{z_0,\rho}).
\end{align*}
For part $(b)$, we have
\begin{align*}
    |D\ell_{z_0,\rho_{j+1}}|
    \le&
    |D\ell_{z_0,\rho_j}|
    +(n+2)\left(\frac{1}{\theta_0}\right)^{n+3}
        \Psi_{1+|D\ell_{z_0,\rho_j}|}(z_0,\rho_j,\ell_{z_0,\rho_j})^\frac{1}{2}\\
    \le&
    M_0-\theta_0^{j\alpha}
    +(n+2)\left(\frac{1}{\theta_0}\right)^{n+3}\theta_0^{j\alpha}\e_0^\frac{1}{2}
    \le
    M_0-\theta_0^{j\alpha}+(1-\theta_0^\alpha)\theta_0^{j\alpha}\\
    =&M_0-\theta^{(j+1)\alpha}.
\end{align*}
By induction, we deduce $(a)$ and $(b)$ for all $j\in\N$.


\noindent
It remains to verify~\eqref{E:CampEst}. Given $0<r\le\rho_0/2$, we may select $j\in\N\cup\{0\}$ such that $\rho_{j+1}<2r\le\rho_j$. Using Remark \ref{mediamin2} and Corollary \ref{cacc3},  we have
\begin{align*}
    \-int_{Q_r(z_0)}\V_{1+|(Du)_{z_0,r}|}(|Du-(Du)_{z_0,r}|)\, d z
    \le&
    \kappa_2\,\-int_{Q_r(z_0)}
        \V_{1+|D\ell_{z_0,\rho_{j}}|}(|Du-D\ell_{z_0,\rho_{j}}|)\, d z\\
    \le&
    \frac{\kappa_2}{\theta_0^{n+2}}\-int_{Q_{\rho_j/2}(z_0)}
        \V_{1+|D\ell_{z_0,\rho_{j}}|}(|Du-D\ell_{z_0,\rho_{j}}|)\, d z\\
    \le&
    c_02^{n+p_1+4}\left(\frac{\kappa_2}{\theta_0^{n+2}}\right)
        \Psi_{1+|D\ell_{z_0,\rho_j}|}(z_0,\rho_j,\ell_{z_0,\rho_j})\\
    \le&
    c_02^{n+p_1+4}\left(\frac{\kappa_2}{\theta_0^{n+2}}\right)
        \theta_0^{2j\alpha}
        \Psi_{1+|D\ell_{z_0,\rho_0}|}(z_0,\rho_0,\ell_{z_0,\rho_0})\\
    \le&
    c_7\left(\frac{r}{\rho_0}\right)^{2\alpha}
        \Psi_{1+|D\ell_{z_0,\rho_0}|}(z_0,\rho_0,\ell_{z_0,\rho_0}).
\end{align*}
Since $\Delta_2(\V,\V^*)$ depends on only $p_1$ and $p_0$, the lemma is proved.
\end{proof}

\section{Partial regularity}
\noindent
We are now in position to prove the main result of the paper.


\medskip  

\noindent {\it Proof of Theorem \ref{main}.}
Let $z_0\in\Omega_T$ be such that
$$
\liminf_{\rho\to 0}\-int_{\Q_\rho(z_0)}|V(Du)-(V(Du))_{z_0,\rho}|^2dz=0,
$$
and
\begin{equation}\label{E:DuBnd}
\limsup_{\rho\to 0}|(Du)_{z_0,\rho}|<+\infty.
\end{equation}
Using \eqref{equiVphi}, we deduce that
$$
\liminf_{\rho\to 0}\-int_{\Q_\rho(z_0)}\V_{1+|(Du)_{z_0,\rho}|}(|Du)-(Du)_{z_0,\rho}|)dz=0.
$$
Exploiting Lemma \ref{quasimin2} and Poincar\'e's inequality in Theorem \ref{poincare},  we get
\begin{equation}\label{exc1}
\begin{split}
&\-int_{\Q_\rho(z_0)}\V_{1+|D\ell_{z_0,\rho}|}\left(\left|\frac{u-\ell_{z_0,\rho}}{\rho}\right|\right)\,dz \\
&\quad \le \kappa_1\-int_{\Q_\rho(z_0)}\V_{1+|(Du)_{z_0,\rho}|}\left(\left|\frac{u-(u)_{z_0,\rho}-(Du)_{z_0,\rho}(x-x_0)}{\rho}\right|\right)\,dz\\
&\quad \le \kappa_1c_2\left[\-int_{\Q_\rho(z_0)}\V_{1+|(Du)_{z_0,\rho}|}(|Du-(Du)_{z_0,\rho}|)\,dz \right. \\
&\qquad \left. +\V_{1+|(Du)_{z_0,\rho}|}\left(\-int_{\Q_\rho(z_0)}\V'_{1+|(Du)_{z_0,\rho}|}(|Du-(Du)_{z_0,\rho}|)\,dz\right)\right].
\end{split}
\end{equation}
Thanks to \eqref{t3} and Jensen's inequality, we may write
\begin{equation}\label{exc11}
\-int_{\Q_\rho(z_0)}\V'_{1+|(Du)_{z_0,\rho}|}(|Du-(Du)_{z_0,\rho}|)\,dz\lesssim (\V^*_{1+|(Du)_{z_0,\rho}|})^{-1}\left(\-int_{\Q_\rho(z_0)}\V_{1+|(Du)_{z_0,\rho}|}(|Du-(Du)_{z_0,\rho}|)\,dz\right).
\end{equation}
On the other hand, Theorem \ref{sp2} implies
\begin{equation}\label{exc2}
\begin{split}
&\-int_{\Q_\rho(z_0)}\left|\frac{u-\ell_{z_0,\rho}}{\rho}\right|^2\,dz\\
&\le \-int_{\Q_\rho(z_0)}\left|\frac{u-(u)_{z_0,\rho}-(Du)_{z_0,2\rho}(x-x_0)}{\rho}\right|^2\,dz\\
&\le c_4 \left[\-int_{\Q_{2\rho}(z_0)}\V_{1+|(Du)_{z_0,2\rho}|}(|Du-(Du)_{z_0,2\rho}|)dz+\left(\-int_{\Q_{2\rho}(z_0)}|Du-(Du)_{z_0,2\rho}|^{p_0}dz\right)^{\frac{2}{p_0}}\right.\\
&\qquad +\V_{1+|(Du)_{z_0,2\rho}|}\left(\-int_{\Q_{2\rho}(z_0)}\V'_{1+|(Du)_{z_0,2\rho}|}(|Du-(Du)_{z_0,2\rho}|)dz\right) \\
&\qquad \left.+\left(\-int_{\Q_{2\rho}(z_0)}\V'_{1+|(Du)_{z_0,2\rho}|}(|Du-(Du)_{z_0,2\rho}|)dz\right)^{2}\right].
\end{split}
\end{equation}
We can use Lemma \ref{L:OrliczPower} to control the upper bound's second integral:
\begin{equation}
\begin{split}\label{exc22}
\-int_{\Q_{2\rho}(z_0)}|Du-(Du)_{z_0,2\rho}|^{p_0}dz&\le C\left(\frac{1}{\V''(1+|(Du)_{z_0,2\rho}|)}\-int_{\Q_{2\rho}(z_0)}\V_{1+|(Du)_{z_0,2\rho}|}(|Du-(Du)_{z_0,2\rho}|)\,dz\right)^{\frac{p_0}{2}}\\
&+C\frac{1}{\V_{1+|(Du)_{z_0,2\rho}|}(1)}\-int_{\Q_{z_0,2\rho}}\V_{1+|(Du)_{z_0,2\rho}|}(|Du-(Du)_{z_0,2\rho}|)\,dz.
\end{split}
\end{equation}
Finally, from \eqref{touse}
\begin{equation}\label{exc3}
\begin{split}
|D\ell_{z_0,\rho}-(Du)_{z_0,\rho}|&\le (n+2)\-int_{\Q_\rho(z_0)}\left|\frac{u-(u)_{z_0,\rho}-(Du)_{z_0,\rho}(x-x_0)}{\rho}\right|\,dz\\
&\lesssim (\V_{1+|(Du)_{z_0,\rho}|})^{-1}\left(\-int_{\Q_\rho(z_0)}\V_{1+|(Du)_{z_0,\rho}|}\left(\left|\frac{u-(u)_{z_0,\rho}-(Du)_{z_0,\rho}(x-x_0)}{\rho}\right|\right)\,dz\right),
\end{split}
\end{equation}
which in turn can be bounded via \eqref{exc1} and \eqref{exc11}. Let $\e_0>0$ be as defined in~\eqref{E:IterParam}. Keeping in mind the definition of $z_0$, the estimates \eqref{exc1}, \eqref{exc2}, and \eqref{exc3}, supported by~\eqref{E:DuBnd}, \eqref{exc11}, and \eqref{exc22}, imply the existence of $M_0> 1$ and a radius $R_0>0$ such that $|D\ell_{z_0,R_0}|<M_0-1$ and $\Psi_{1+|D\ell_{z_0,R_0}|}(z_0,R_0,\ell_{z_0,R_0})<\e_0$. By the absolute continuity of the integrals, there exists $R_1<R_0$ such that, for any $z\in\Q_{R_1}(z_0)$ we have 
$$
1+|D\ell_{z,R_0}|<M_0\ \ \ \ \ \ \hbox{ and }\ \ \ \ \ \ \Psi_{1+|D\ell_{z,R_0}|}(z,R_0,\ell_{z,R_0})< {\e_0}.
$$
Applying Lemma \ref{iteration} to each point $z\in\Q_{R_1}(z_0)$, we deduce that, for any $r\le R_0/2$,
\[
\int_{\Q_r(z)}|V(Du)-(V(Du))_{z,r}|^2dz\sim\int_{\Q_r(z)}\V_{1+|(Du)_{z,r}|}(|Du-(Du)_{z,r}|)dz\le  C(M,\theta_0)\, \frac{r^{n+2+2\alpha}}{R_0^{2\alpha}}\e_0.
\]
This means that $V(Du)$ belongs to the parabolic Campanato space ${\mathcal L}^{2,\frac{2\alpha}{n+2}}(\Q_{R_1}(z_0),\R^{Nn})$ and by the usual embedding we have $\displaystyle{V(Du)\in C^{0,\frac{\alpha}{2},\alpha}(\Q_{R_1}(z_0),\R^{Nn})}$.
\qed

\begin{remark}
Note, as indicated in Section~\ref{S:Intro}, the H\"older continuity of $V(Du)$ implies the H\"older continuity of $Du$ with a different exponent depending on $\V$. 
\end{remark}

\end{document}